\documentclass[12pt]{amsart}
\usepackage{amsmath}
\usepackage{amsfonts}
\usepackage{amssymb}
\usepackage[all,cmtip]{xy}           
\usepackage{bm}
\usepackage{bbm}
\usepackage{bbding}
\usepackage{txfonts}
\usepackage{amscd}
\usepackage{xspace}
\usepackage[shortlabels]{enumitem}
\usepackage{ifpdf}
\usepackage{ulem}

\ifpdf
  \usepackage[colorlinks,final,backref=page,hyperindex]{hyperref}
\else
  \usepackage[colorlinks,final,backref=page,hyperindex,hypertex]{hyperref}
\fi
\usepackage{tikz}
\usepackage[active]{srcltx}

\usepackage{tikz-cd}
\usepackage{tikz}

\makeatletter

\newtheorem{thm}{Theorem}[section]
\newtheorem{lem}[thm]{Lemma}
\newtheorem{cor}[thm]{Corollary}
\newtheorem{pro}[thm]{Proposition}
\theoremstyle{definition}
\newtheorem{ex}[thm]{Example}
\newtheorem{rmk}[thm]{Remark}
\newtheorem{defi}[thm]{Definition}

\newcommand{\nc}{\newcommand}
\newcommand{\delete}[1]{}

\nc{\mlabel}[1]{\label{#1}}  
\nc{\mcite}[1]{\cite{#1}}  
\nc{\mref}[1]{\ref{#1}}  
\nc{\meqref}[1]{\eqref{#1}}  
\nc{\mbibitem}[1]{\bibitem{#1}} 

\delete{
\nc{\mlabel}[1]{\label{#1}{\hfill \hspace{1cm}{\bf{{\ }\hfill(#1)}}}}
\nc{\mcite}[1]{\cite{#1}{{\bf{{\ }(#1)}}}}  
\nc{\mref}[1]{\ref{#1}{{\bf{{\ }(#1)}}}}  
\nc{\meqref}[1]{\eqref{#1}{{\bf{{\ }(#1)}}}}  
\nc{\mbibitem}[1]{\bibitem[\bf #1]{#1}} 
}

\newcommand {\emptycomment}[1]{}

\newcommand{\mg}[1]{\textcolor{red}{#1}}

\nc{\oprn}{\theta}

\nc{\calo}{\mathcal{O}}
\nc{\oop}{$\mathcal{O}$-operator\xspace}
\nc{\oops}{$\mathcal{O}$-operators\xspace}
\nc{\mrho}{{\bm{\varrho}}}
\nc{\bfk}{\mathbf{K}}
\nc{\invlim}{\displaystyle{\lim_{\longleftarrow}}\,}
\nc{\ot}{\otimes}

\nc{\eval}[1]{\Big|_{#1}}

\newcommand{\add}{\frka\frkd}

\newcommand{\lon }{\,\rightarrow\,}
\newcommand{\be }{\begin{equation}}
\newcommand{\ee }{\end{equation}}

\newcommand{\g}{\mathfrak g}
\newcommand{\h}{\mathfrak h}

\nc{\RR}{\mathbb{R}}

\nc{\CC}{\mathbb{C}}

\newcommand{\huaR}{\mathcal{R}}

\newcommand{\huaG}{\mathcal{G}}

\newcommand{\huaI}{\mathcal{I}}

\newcommand{\frka}{\mathfrak a}

\newcommand{\frkd}{\mathfrak d}

\newcommand{\frkG}{\mathfrak G}

\newcommand{\frkR}{\mathfrak R}

\newcommand{\Courant}[1]{\left\llbracket  #1\right\rrbracket }

\newcommand{\Id}{{\rm{Id}}}

\newcommand{\br}[1]{   [ \cdot,    \cdot  ]   }

\newcommand{\Der}{\mathrm{Der}}

\newcommand{\Ad}{\mathrm{Ad}}

\newcommand{\gl}{\mathfrak {gl}}

\newcommand{\ad}{\mathrm{ad}}

\nc{\CV}{\mathbf{C}}

\begin{document}

\title[Reynolds Lie bialgebras]{Reynolds Lie bialgebras}

\author{Shuai Hou}
\address{Department of Mathematics, Jilin University, Changchun 130012, Jilin, China}
\email{hshuaisun@jlu.edu.cn}

\author{Maxim Goncharov}
\address{Sobolev Institute of Mathematics
Acad. Koptyug ave. 4, 630090 Novosibirsk, Russia
Novosibirsk State University
Pirogova str. 2, 630090 Novosibirsk, Russia}
\email{gme@math.nsc.ru}


\begin{abstract}
In this paper, we establish a bialgebra theory for Reynolds Lie algebras. First we introduce the notion of a quadratic Reynolds Lie algebra and show that it induces an isomorphism from the adjoint representation to the coadjoint representation. Then we introduce the notion of matched pairs, Manin triples and bialgebras for Reynolds Lie algebras, and show that Manin triples, bialgebras and certain matched pairs of Reynolds Lie algebras are equivalent. In particular, we introduce the notion of a Reynolds operator on a quadratic Rota-Baxter Lie algebra which can induce a Reynolds Lie bialgebra naturally.
Finally, we introduce the notion of the classical Yang-Baxter equation in a Reynolds Lie algebra whose solutions give rise to Reynolds Lie bialgebras. We also introduce the notion of relative Rota-Baxter operators on a Reynolds Lie algebra and Reynolds pre-Lie algebras, and construct solutions of the classical Yang-Baxter equation in terms of relative Rota-Baxter operators and Reynolds pre-Lie algebras.
\end{abstract}


\keywords{Reynolds operator, Reynolds Lie bialgebra, Manin triple, classical Yang-Baxter equation\\
\quad  2020 {\it Mathematics Subject Classification.}  17B38,  17B62}

\maketitle

\tableofcontents

\section{Introduction}
The aim of this paper is to establish a bialgebra theory for Reynolds Lie algebras, and construct Reynolds Lie bialgebras by using (relative) Rota-Baxter operators and skew-symmetric solutions of the classical Yang-Baxter equation in Reynolds Lie algebras.

\subsection{Reynolds operators}
Reynolds operators originated from O. Reynolds' famous research on the theory of turbulence in fluid mechanics;
the Reynolds operator plays an important role in the Navier-Stokes equations \cite{Reynolds}.
 G. Birkhoff \cite{Birkhoff} and G. - C. Rota \cite{Rota}
 have also carried out a series of studies about Reynolds operators.
 In addition,  Reynolds operators are closely related to the probability theory
 developed by Kolmogorov, especially conditional expectation.
 In \cite{ZGG-1}, the authors studied the relevant properties of Reynolds operators
 on the associative algebra and the construction of free Reynolds algebras.
 Also, they introduced the concept of Reynolds words and verified that
 Reynolds words have the algebraic structure and universal properties of
 free Reynolds algebras, see \cite{ZGG-2} for more details about
 the construction of free Reynolds algebras.
 Recently, A. Das introduced the concept of Reynolds operators
 on Lie algebras and established the relationship between Reynolds
 operators and twisted Rota-Baxter operators \cite{Das}.
 In \cite{Guo-Gus-Li}, the authors introduced the concepts of weighted Reynolds operators
 and differential Reynolds operators, established the relationship between separable
  Volterra operators and weighted Reynolds operators, and studied Volterra integral operators
  and their integral equations from an algebraic perspective.

\subsection{Lie bialgebras}
 Lie bialgebras appeared in the paper of  V. G. Drinfel'd as a tool to study solutions of the classical Yang-Baxter equations \cite{Drinfeld-cybe}.
A Lie algebra $(\g, [\cdot,\cdot]_{\g})$, where $[\cdot,\cdot]_{\g}:\otimes^2\g\rightarrow\g$ is a Lie bracket, equipped with a comultiplication $\Delta:\g\rightarrow\g\otimes\g$ is called a Lie bialgebra, if the dual algebra of the coalgebra $(\g,\Delta)$ is a Lie algebra and the comultiplication $\Delta$ is a 1-cocycle.

Lie bialgebras have important applications in both mathematics and mathematical physics. A Lie bialgebra is the algebraic structure corresponding to a Poisson-Lie group and the classical structure of a quantized universal enveloping algebra \cite{CP,Drinfeld-cybe}. In the finite-dimensional case, the category of Lie bialgebras is equivalent to the category of Manin triples of Lie algebras \cite{Drinfeld-Quantum}.

 P. Etingof and D. Kazhdan proposed a general method to quantize a Lie bialgebra into a quantum group, and the quantization process preserves functoriality \cite{EK}. Y. Kosmann-Schwarzbach established the relationship between factorizable Lie bialgebras and double Lie algebras \cite{K}.

\subsection{The classical Yang-Baxter equation and Rota-Baxter operators}
The Yang-Baxter equation was independently discovered by C. N. Yang in his study of the two-dimensional Ising model in 1967 \cite{Yang} and by R. J. Baxter in his research on the eight-vertex model in 1972 \cite{R.J.Baxter}.
The Yang-Baxter equation has been extensively studied and applied in various fields of mathematical physics, such as integrable systems, quantum groups and knot theory. The classical Yang-Baxter equation can be considered as the semi-classical limit of the quantum Yang-Baxter equation. The connection between the classical Yang-Baxter equation and Lie bialgebras constitute the core of quantum group theory. In particular, solutions of the classical Yang-Baxter equation give rise to certain Lie bialgebras \cite{Drinfeld-cybe,RS}.

The Rota-Baxter operator on associative algebras can be traced back from G. Baxter's research in probability theory \cite{G. Baxter}. Rota-Baxter operators have important applications in the fields of mathematics and physics, including renormalization of quantum field theory \cite{CK}, integrable system \cite{BGN2010,STS} and double Poisson algebras \cite{Goncharov-Kolesnikov,Goncharov-Gubarev}. See \cite{Gub} for more details. M.A. Semenov-Tian-Shansky established  a close relationship between Rota-Baxter operators on  quadratic Lie algebras and solutions of the  classical Yang-Baxter equation. Namely, he established a one-to-one correspondence between Rota-Baxter operators of weight 0 and skew-symmetric solutions of the classical Yang-Baxter equation (in other words, triangular Lie bialgebra structures), and between Rota-Baxter operators of weight $1$ and solutions of the (modified) classical Yang-Baxter equation \cite{STS}. A connection between Rota-Baxter operators of weight 1 (of special type) and structures of factorizable Lie bialgebras on a quadratic Lie algebra was studied in \cite{RS,Lang-Sheng,MG}.

   B.A. Kupershmidt proposed the concept of relative Rota-Baxter operators on Lie algebras as the operator form of the classical Yang-Baxter equation \cite{Ku}. C. Bai \cite{Bai} showed that a relative Rota-Baxter operator can give a skew-symmetric solution of the classical Yang-Baxter equation in the semidirect Lie algebra.
 For more applications of (relative) Rota-Baxter operators in Lie algebras and the classical Yang-Baxter equation, please refer to \cite{BD}.

\subsection{Main results and outline of the paper}
 Taking into account the great importance of Reynolds operators and Lie bialgebras, in this paper we focus on the study of Reynolds operators on classical doubles of Lie bialgebras. This naturally led us to the definition of a Reynolds Lie bialgebra and motivated us to develop a suitable bialgebra theory for Reynolds Lie bialgebras.

To develop a bialgebra theory for Reynolds Lie algebras, first we study  representations and invariant bilinear forms on Reynolds Lie algebras. In particular, the dual representation associated to a representation exists naturally, which guarantees that there will be a good bialgebra theory in some sense. We introduce the notion of quadratic Reynolds Lie algebras and show that the invariant bilinear form induces an isomorphism from the adjoint representation to the coadjoint representation of the Reynolds Lie algebra.  Then we introduce the notions of matched pairs and Manin triples of Reynolds Lie algebras. Thanks to  the existence of the coadjoint representation, we show that  Manin triples of Reynolds Lie algebras and matched pairs associated to the coadjoint representations are equivalent. We further introduce the notion of a Reynolds Lie bialgebra and show that a Reynolds operator on a quadratic Rota-Baxter Lie algebra naturally gives rise to a Reynolds Lie bialgebra. Finally, we study the classical Yang-Baxter equation in a Reynolds Lie algebra and introduce the notions of relative Rota-Baxter operators on Reynolds Lie algebras and Reynolds pre-Lie algebras, by which we construct solutions of the classical Yang-Baxter equation in a certain Reynolds Lie algebra. We summarize the main results in this paper by the following diagram:

 
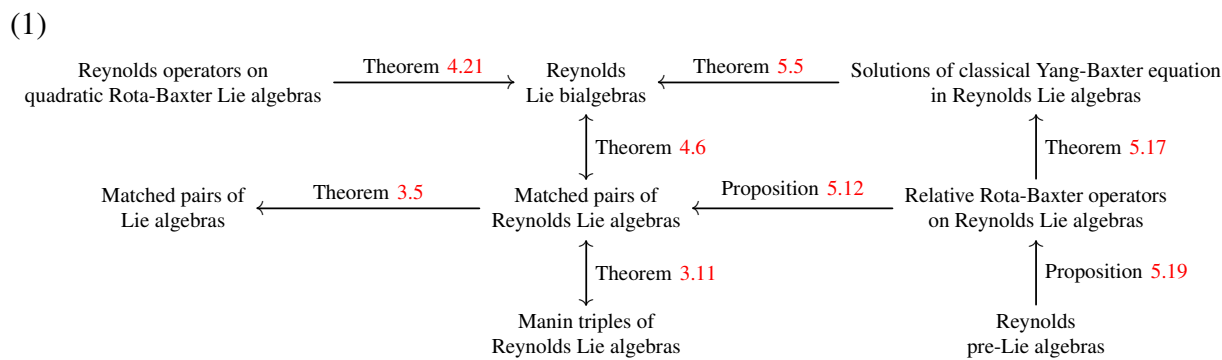
\begin{equation}
    \begin{split}
\xymatrix{
   \text{Reynolds operators on}\atop \text{quadratic Rota-Baxter Lie algebras}\ar[rr]^-{\rm Theorem~\ref{thm:FL}}&& \text{Reynolds}\atop \text{Lie bialgebras} \ar@{<->}_{}[d]^-{\rm Theorem~\ref{bialgebra-matched}}\ar@{<-}_{}[rr]^-{\rm Theorem~\ref{Rey-bialgebra-cybe}} && \text{Solutions of classical Yang-Baxter equation}\atop \text{in Reynolds Lie algebras} \ar@{<-}^{}[d]^-{\rm Theorem~\ref{skew-solution-Reynolds}} \\
 \text{Matched pairs of}\atop \text{Lie algebras}  \ar@{<-}^{}[rr]^-{\rm Theorem~\ref{demp}} & & \text{Matched pairs of}\atop\text{Reynolds Lie algebras}\ar@{<->}_{}[d]^-{\rm Theorem~\ref{matched-manin-equivalent}}\ar@{<-}^{}[rr]^-{\rm Proposition~\ref{construction-mathched-pair}} && \text{Relative Rota-Baxter operators}\atop \text{on Reynolds Lie algebras} \ar@{<-}^{}[d]^-{\rm Proposition~\ref{Reynolds-Pre-conclusion3}} \\
  & & \text{Manin triples of}\atop \text{Reynolds Lie algebras} & & \text{Reynolds}\atop \text{pre-Lie algebras}  }
\\
    \end{split}
\end{equation}

The paper is organized as follows. In Section \ref{sec:rep}, we study representations
and quadratic Reynolds Lie algebras.  We show that the invariant bilinear form in
a quadratic Reynolds Lie algebra naturally induces an isomorphism from the adjoint
representation to the coadjoint representation (Theorem \ref{quadratic-Rep}).
In Section \ref{sec:mp},   we introduce the notions of matched pairs of Reynolds Lie algebras and
 Manin triples of Reynolds Lie algebras, and show that Manin triples of Reynolds
 Lie algebras and certain matched pairs of Reynolds Lie algebras are
 equivalent (Theorem \ref{matched-manin-equivalent}). In Section \ref{sec:bia},
 first we introduce the notion of a Reynolds Lie bialgebra,
 and show that they are equivalent to certain matched pairs of Reynolds Lie algebras (Theorem \ref{bialgebra-matched}). We show that an Reynolds operator on a quadratic Rota-Baxter Lie algebra induces a Reynolds Lie bialgebra naturally (Theorem \ref{thm:FL}). In Section \ref{sec:YBE}, we introduce the notion of the classical Yang-Baxter equation in a Reynolds Lie algebra and  construct solutions   in terms of relative Rota-Baxter operators on Reynolds Lie algebras and Reynolds pre-Lie algebras.
\vspace{2mm}
\noindent

Unless otherwise is specified, in this paper all vector spaces are assumed to be finite-dimensional over a field $\mathbb F$ of characteristic zero.

\section{Reynolds Lie algebras}\label{sec:rep}
In this section, we introduce the notion of a representation of a Reynolds Lie algebra, and show that the dual representation exists naturally.  We also introduce the notion of a quadratic Reynolds Lie algebra, and show that it induces an isomorphism from the adjoint representation and the coadjoint representation. \emptycomment{ A representation of a Reynolds Lie algebra also induces a representation of the underlying NS-Lie algebra.}
\begin{defi}\rm{(\cite{Das})}
Let $(\g,[\cdot,\cdot]_\g)$ be a Lie algebra. A linear operator $R:\g\rightarrow \g$ is called a
{\bf Reynolds operator} if
\begin{equation}\label{Reynolds-operator}
 [Rx,Ry]_{\g}=R\Big([Rx,y]_{\g}+[x,Ry]_{\g}-[Rx,Ry]_{\g}\Big), \quad \forall x, y \in \g.
\end{equation}
Moreover, a Lie algebra $\g$ equipped with a Reynolds operator $R$ is called a {\bf Reynolds Lie algebra},
which is denoted by $(\g,[\cdot,\cdot]_{\g},R).$
\end{defi}

\begin{defi}
Let $(\g,[\cdot,\cdot]_\g,R)$ and $(\g',[\cdot,\cdot]_{\g'},R')$ be Reynolds Lie algebras. A homomorphism from $(\g,[\cdot,\cdot]_\g,R)$ to $(\g',[\cdot,\cdot]_{\g'},R')$ consists of a Lie algebra homomorphism $\phi:\g\rightarrow\g'$ such that
\begin{eqnarray*}
\phi\circ R&=&R'\circ\phi.
\end{eqnarray*}
Furthermore, if $\phi$ is invertible, then $\phi$ is called an {\bf isomorphism}.
\end{defi}

\begin{pro}\rm{(\cite{FoxMil})}\label{dre-Rey}
    Let $(\g,[\cdot,\cdot]_{\g})$ be a Lie algebra and $d\in\Der(\g)$ is a derivation of $\g$ such that the map $T=\Id-d$ is invertible. Then $R=T^{-1}=(\Id-d)^{-1}$ is a Reynolds operator on $\g$.
\end{pro}
\begin{ex}\label{dre-Rey-ex}
    Consider the $3$-dimensional simple Lie algebra $(\g,[\cdot,\cdot]_{\g})=sl(2,\CC)$ with a basis
\begin{equation*}
\left\{x=\left(
       \begin{array}{cc}
         0 & 1 \\
         0 & 0 \\
       \end{array}
     \right),\quad y=\left(
                          \begin{array}{cc}
                            0 & 0 \\
                            1 & 0 \\
                          \end{array}
                        \right), \quad h =\left(
                          \begin{array}{cc}
                            1 & 0 \\
                            0 & -1 \\
                          \end{array}
                        \right)\right\}.
\end{equation*}
The multiplication is given by
\begin{eqnarray*}
[h,x]_{\g}=2x,\quad [h,y]_{\g}=-2y,\quad [x,y]_{\g}=h.
\end{eqnarray*}
Consider $d=\ad_{h}$, the inner derivation acting as $d(a)=[h,a]$. It is straightforward to check that $T=\Id-\ad_h$ is invertible and the map $R=T^{-1}$ acts as
$$
R(h)=h,\ R(x)=-x,\ R(y)=\frac{1}{3}y.
$$
By Proposition \ref{dre-Rey}, $R$ is a Reynolds operator on $sl(2,\mathbb C)$. Note that maps $\Id-\ad_x$ and $\Id-\ad_y$ are also invertible. Therefore, $(\Id-\ad_x)^{-1}$ and $(\Id-\ad_y)^{-1}$ are also Reynolds operators on $sl(2,\mathbb C)$.
\end{ex}

\begin{pro}\label{induce-Lie-algebra}
Let  $(\g,[\cdot,\cdot]_{\g},R)$ be a Reynolds Lie algebra. Then $R$ induces a new Lie algebra structure
on $\g$ given by
\begin{eqnarray}
  [x,y]_{R}:=[Rx,y]_{\g}+[x,Ry]_{\g}-[Rx,Ry]_{\g},\quad \forall x,y\in \g.
\end{eqnarray}
The Lie algebra $(\g,[\cdot,\cdot]_{R})$ is called the {\bf induced Lie algebra,} and denoted by $\g_{R}.$ Moreover, $(\g,[\cdot,\cdot]_{R},R)$ is also a Reynolds Lie algebra and $R$ is a homomorphism from Reynolds Lie algebra $(\g,[\cdot,\cdot]_{R},R)$ to $(\g,[\cdot,\cdot]_\g,R).$
\end{pro}
\begin{proof} For all $x,y,z\in \g,$ it is obvious that $[Rx,Ry]_{\g}=R[x,y]_{R},$ and
\begin{eqnarray*}
&&[x,[y,z]_{R}]_{R}-[[x,y]_{R},z]_{R}-[y,[x,z]_{R}]_{R}\\
&=&[Rx,[Ry,z]_{\g}]_{\g}+[Rx,[y,Rz]_{\g}]_{\g}-[Rx,[Ry,Rz]_{\g}]_{\g}+[x,[Ry,Rz]_{\g}]_{\g}-[Rx,[Ry,Rz]_{\g}]_{\g}\\
&&-[[Rx,Ry]_{\g},z]_{\g}-[[Rx,y]_{\g},Rz]_{\g}-[[x,Ry]_{\g},Rz]_{\g}+[[Rx,Ry]_{\g},Rz]_{\g}+[[Rx,Ry]_{\g},Rz]_{\g}\\
&&-[Ry,[Rx,z]_{\g}]_{\g}-[Ry,[x,Rz]_{\g}]_{\g}+[Ry,[Rx,Rz]_{\g}]_{\g}-[y,[Rx,Rz]_{\g}]_{\g}+[Ry,[Rx,Rz]_{\g}]_{\g}\\
&=&0,
\end{eqnarray*}
which implies that $(\g,[\cdot,\cdot]_{R})$ is a Lie algebra. Furthermore, we have
\begin{eqnarray*}
  [Rx,Ry]_{R}&=&[R^2x,Ry]_{\g}+[Rx,R^2y]_{\g}-[R^2x,R^2y]_{\g}\\
  &=&R\Big([Rx,y]_{R}+[x,Ry]_{R}-[Rx,Ry]_{R}\Big),
\end{eqnarray*}
which implies that $(\g,[\cdot,\cdot]_{R},R)$ is a Reynolds Lie algebra and  $R$ is a homomorphism from $(\g,[\cdot,\cdot]_{R},R)$ to $(\g,[\cdot,\cdot]_\g,R).$
\end{proof}

\begin{ex}
Let $q\in \mathbb C$ and $B(q)$ be the Block Lie algebra \rm{(\cite{Xia-You-Zhou})} generated by basis elements $\{L_{m,i}|m,i\in \mathbb Z\},$ where the Lie bracket given
 by
\begin{eqnarray*}
[L_{m,i}, L_{n,j}]_{\g} = (n(i+q)-m(j+q))L_{m+n,i+j}, \quad \forall i,j,m,n \in  \mathbb Z.
\end{eqnarray*}
Then the linear map $R:B(q)\rightarrow B(q)$ defined by $R(L_{m,i})=\frac{1}{m+i+1}L_{m,i}$ is a Reynolds operator on $B(q).$ On the one hand, for any $m,n,i,j\in \mathbb Z,$
\begin{eqnarray}\label{Rey-left}
  [R(L_{m,i}),R(L_{n,j})]_{\g}=\left[\frac{1}{m+i+1}L_{m,i},\frac{1}{n+j+1}L_{n,j}\right]_{\g}=\frac{(n(i+q)-m(j+q))}{(m+i+1)(n+j+1)}L_{m+n,i+j}.
\end{eqnarray}
On the other hand,
\begin{eqnarray}\label{Rey-right}
\nonumber&&R[R(L_{m,i}),L_{n,j}]_{\g}+R[L_{m,i},R(L_{n,j})]_{\g}-R[R(L_{m,i}),R(L_{n,j})]_{\g}\\
\nonumber&=&\Big(\frac{(n(i+q)-m(j+q))}{(m+i+1)(m+n+i+j+1)}+\frac{(n(i+q)-m(j+q))}{(n+j+1)(m+n+i+j+1)}\\
\nonumber&&-\frac{(n(i+q)-m(j+q))}{(m+i+1)(n+j+1)(m+n+i+j+1)}\Big)L_{m+n,i+j}\\
&=&\frac{(n(i+q)-m(j+q))}{(m+i+1)(n+j+1)}L_{m+n,i+j}.
\end{eqnarray}
It follows from \eqref{Rey-left} and \eqref{Rey-right} that $(B(q),[\cdot,\cdot]_{\g},R)$ is a Reynolds Lie algebra. The induced Reynolds Lie algebra structure $(B(q),[\cdot,\cdot]_{R},R)$ is given by
\begin{eqnarray*}
  [L_{m,i},L_{n,j}]_{R}:&=&[R(L_{m,i}),L_{n,j}]_{\g}+[L_{m,i},R(L_{n,j})]_{\g}-[R(L_{m,i}),R(L_{n,j})]_{\g}\\
  &=&\frac{(m+n+i+j+1)((n(i+q)-m(j+q)))}{(m+i+1)(n+j+1)}L_{m+n,i+j}.
\end{eqnarray*}
\end{ex}
\begin{defi}
    Let $\rho:\g\rightarrow\gl(W)$ be a representation of a Lie algebra $(\g,[\cdot,\cdot]_{\g})$. Then there is a natural Lie algebra structure on the semidirect product $\g\oplus W$ given by
    \begin{eqnarray*}
   [x+u,y+v]_{\ltimes}=[x,y]_{\g}+\rho(x)v-\rho(y)u,\quad \forall x,y\in \g,u,v\in W,
    \end{eqnarray*}
\end{defi}
\begin{pro}\label{semi-direct-ReyLiealge}
 Let $(W;\rho)$ be a representation of a Lie algebra $(\g,[\cdot,\cdot]_{\g})$, $R$ be a Reynolds operator on $\g$ and $T:W\rightarrow W$ be a linear map. Define a map $\widehat{R}:\g\oplus W\rightarrow\g\oplus W$ by
   \begin{eqnarray*}
   \widehat{R}(x+u)=Rx+Tu,\quad \forall x\in \g, u\in W.
    \end{eqnarray*}
    Then $(\g\oplus W,[\cdot,\cdot]_{\ltimes},\widehat{R})$ is a Reynolds Lie algebra if and only if for all $x\in \g$ and $u\in W$
    \begin{equation}\label{tmp1}\rho(Rx)(Tu)=T\Big(\rho(x)(Tu)+\rho(Rx)(u)-\rho(Rx)(Tu)\Big).\end{equation}
\end{pro}

\begin{proof}
It is known that  $(\g\oplus W, [\cdot,\cdot]_{\ltimes})$ is a Lie algebra.
For all $x,y \in\g, u,v\in W$,  we have
 \begin{eqnarray*}
   &&[\widehat{R}(x+u),\widehat{R}(y+v)]_{\ltimes}-\widehat{R}\Big([\widehat{R}(x+u),y+v]_{\ltimes}
   +[x+u,\widehat{R}(y+v)]_{\ltimes}-[\widehat{R}(x+u),\widehat{R}(y+v)]_{\ltimes}]\Big)\\
   &=&[Rx,Ry]_{\g}+\rho(Rx)(Tv)-\rho(Ry)(Tu)-R[x,Ry]_{\g}-T(\rho(x)(Tv))+T(\rho(Ry)u)\\
   &&-R[Rx,y]_{\g}-T(\rho(Rx)v)+T(\rho(y)(Tu))+R[Rx,Ry]_{\g}+T(\rho(Rx)(Tv))-T(\rho(Ry)(Tu))\\
   &=&\rho(Rx)(Tv)-T(\rho(x)(Tv))-T(\rho(Rx)v)+T(\rho(Rx)(Tv))\\
   && -\rho(Ry)(Tu)+T(\rho(Ry)u)+T(\rho(y)(Tu))-T(\rho(Ry)(Tu)).
 \end{eqnarray*}
Thus, if \eqref{tmp1} is satisfied, then $\widehat{R}$ is a Reynolds operator on $\g\oplus W$.
Conversely, if $\widehat{R}$ is a Reynolds operator
 on $\g\oplus W$, then taking $u=v$ and $y=0$ will lead us to \eqref{tmp1}.
\end{proof}

This naturally leads us to the notion of representations of Reynolds Lie algebras.

\begin{defi}
Let $(\g,[\cdot,\cdot]_{\g},R)$ be a Reynolds Lie algebra.
A {\bf representation} of $(\g,[\cdot,\cdot]_{\g},R)$ on a vector space $W$ is a pair $(T,\rho)$, where $\rho:\g\rightarrow\gl(W)$ is a representation of the Lie algebra $\g$ on $W$, and $T:W\rightarrow W$ is a linear map satisfying
\begin{equation}\label{representation-Reynolds-Liealg}
\rho(Rx)(Tu)=T\Big(\rho(x)(Tu)+\rho(Rx)(u)-\rho(Rx)(Tu)\Big), \quad \forall x\in\g, u\in W.
\end{equation}
\end{defi}
We will denote a representation by $(W;T,\rho).$

\begin{ex}\label{adjoint-representation-Rey}
  It is straightforward to see that $(\g;R,\ad)$ is a representation of a Reynolds Lie algebra $(\g,[\cdot,\cdot]_{\g},R)$, which is called the {\bf adjoint representation}.
\end{ex}

Next we consider the {\bf dual representation} of a Reynolds Lie algebra. Let $(W;T,\rho)$ be a representation of a Reynolds Lie algebra $(\g,[\cdot,\cdot]_{\g},R).$ Define $\rho^{*}:\g\rightarrow\gl(W^*)$ by
\begin{eqnarray}
 \langle \rho^{*}(x)\xi,u \rangle=-\langle\xi,\rho(x)u\rangle,\quad \forall x\in\g, u\in W,\xi\in W^{*}.
\end{eqnarray}
It is well known that $\rho^{*}$ is a representation of $\g$ on $W^{*}.$
Consider the operator $T^{*}:W^{*}\rightarrow W^{*},$ where $T^{*}$ is the usual dual of $T$, i.e.
\begin{eqnarray}\label{dual-Reynolds-operator}
 \langle T^{*}(\xi),u\rangle=\langle\xi,Tu\rangle,\quad \forall u\in W, \xi \in W^{*}.
\end{eqnarray}

\begin{pro}\label{dual-representation}
  With the above notations, $(W^{*};-T^{*}, \rho^{*})$ is a representation of the Reynolds Lie algebra $(\g,[\cdot,\cdot]_{\g},R),$ called the {\bf dual representation.}
\end{pro}
\begin{proof}
  Since $\rho^{*}$ is a representation of $\g$, we only need to prove the following equality:
  \begin{eqnarray*}\label{dual-representation-proof}
  \rho^{*}(Rx)(-T^{*}\xi)=T^*\Big(\rho^*(x)(T^*\xi)-\rho^*(Rx)(\xi)-\rho^*(Rx)(T^*\xi)\Big),\quad \forall x\in \g,\xi\in W^*.
  \end{eqnarray*}

For all $x\in \g, u\in W, \xi\in W^{*},$ we have
\begin{eqnarray*}
&&\langle \rho^{*}(Rx)(-T^{*}\xi)-T^*(\rho^*(x)(T^*\xi))+T^*(\rho^*(Rx)(\xi))+T^*(\rho^*(Rx)(T^*\xi)),u\rangle\\
&=&\langle T^*\xi,\rho(Rx)u\rangle-\langle\rho^*(x)(T^*\xi),Tu\rangle+\langle\rho^*(Rx)(\xi),Tu \rangle+\langle\rho^*(Rx)(T^*\xi),Tu\rangle\\
&=&\langle\xi,T(\rho(Rx)u)+T(\rho(x)(Tu))-\rho(Rx)(Tu)-T(\rho(Rx)(Tu))\rangle\\
&=&0;
\end{eqnarray*}
which implies that \eqref{representation-Reynolds-Liealg} holds. Therefore,  $(W^{*};-T^{*}, \rho^{*})$ is a representation of $(\g,[\cdot,\cdot]_{\g},R).$
\end{proof}

\begin{cor}\label{thm:coadjoint}
Let $(\g,[\cdot,\cdot]_{\g},R)$ be a Reynolds Lie algebra. Then $(\g^*; -R^*,\ad^*)$ is a representation of $(\g,[\cdot,\cdot]_{\g},R)$, which is called the {\bf coadjoint representation}.
\end{cor}
\begin{proof}
  It follows from  Example \ref{adjoint-representation-Rey} and Proposition \ref{dual-representation} directly.
\end{proof}

Next we introduce the notion of a quadratic Reynolds Lie algebra, and show that the invariant bilinear form induces an isomorphism from the adjoint representation to the coadjoint representation of a Reynolds Lie algebra. Quadratic Reynolds Lie algebras will also play important roles in our later study of Manin triples of Reynolds Lie algebras.

Recall that a nondegenerate symmetric bilinear form $S\in \otimes ^2\g^*$ on a Lie algebra $\g$ is said to be invariant if
\begin{eqnarray}
\label{Invariant}S([x,y]_\g,z)+S(y,[x,z]_\g)&=&0,\quad \forall x,y\in \g.
\end{eqnarray}
A quadratic Lie algebra $(\g,S)$ is a Lie algebra $\g$ equipped with a nondegenerate symmetric invariant bilinear form $S\in \otimes ^2\g^*$.

\begin{defi}\label{defi:qua}
Let $(\g,[\cdot,\cdot]_\g,R)$ be a Reynolds Lie algebra, and $S\in \otimes ^2\g^*$ be a nondegenerate symmetric bilinear form.  The triple  $(\g,R,S)$ is called a {\bf quadratic Reynolds Lie algebra} if $(\g,S)$ is a quadratic Lie algebra and the following compatibility condition holds:
\begin{eqnarray}
\label{Rey-manin}
S(Rx,y)+S(x, Ry)&=&0,\quad \forall x,y\in \g.
\end{eqnarray}
\end{defi}

\begin{defi}
 Let $(W;T,\rho)$ and $(W';T',\rho')$ be two representations of a Reynolds Lie algebra $(\g,[\cdot,\cdot]_{\g},R)$. A homomorphism from $(W; T,\rho)$ to $(W'; T',\rho')$ is a linear map $\phi:W\to W'$ such that
 \begin{eqnarray*}
   \phi\circ \rho(x)&=&\rho'(x)\circ \phi,\quad \forall x\in\g,\\
   \phi\circ T&=&T'\circ \phi.
 \end{eqnarray*}
\end{defi}

\begin{thm}\label{quadratic-Rep}
If  $(\g,R,S)$ is a quadratic Reynolds Lie algebra, then the linear map $S^\sharp:\g\to\g^*$ defined by
 $
 \langle S^\sharp(x),y\rangle=S(x,y)
 $
for all $x,y\in\g$, is an isomorphism from the adjoint representation $(\g;R,\ad)$ to the coadjoint representation $(\g^*;-R^*,\ad^*)$.
\end{thm}

\begin{proof}
   Let  $(\g,R,S)$ be a quadratic Reynolds Lie algebra. By \eqref{Invariant}, we have
  $$
      S^\sharp\circ \ad_x=\ad_x^*\circ S^\sharp.
    $$
By \eqref{Rey-manin}, we have
\begin{eqnarray*}
 S(Rx,y)+S(x,Ry)
&=&\langle S^\sharp(Rx),y\rangle+\langle S^\sharp(x),R(y)\rangle\\
&=&\langle S^\sharp(Rx),y\rangle+\langle R^{*}(S^\sharp(x)),y\rangle\\
&=&0,
\end{eqnarray*}
that is
\begin{eqnarray}\label{eq:SP-1}
S^\sharp\circ R= -R^*\circ S^\sharp.
\end{eqnarray}
 Therefore, $S^\sharp:\g\to\g^*$ is an isomorphism from the adjoint representation $(\g;R,\ad)$ to the coadjoint representation $(\g^*;-R^*,\ad^*)$.
\end{proof}

\section{Matched pairs and Manin triples of Reynolds Lie algebras}\label{sec:mp}

In this section,   we introduce the notion of  matched pairs of Reynolds Lie algebras and Manin triples of Reynolds Lie algebras, and show that Manin triples of Reynolds Lie algebras and certain matched pairs of Reynolds Lie algebras are equivalent.

\subsection{Matched pairs of Reynolds Lie algebras}
In this subsection,  we introduce the notion of  matched pairs of Reynolds Lie algebras and show that a matched pair of Reynolds Lie algebras gives rise to an induced matched pair of Lie algebras.

Matched pairs of Lie algebras are also known as  twilled Lie algebras \rm{(\cite{KM})} or double Lie algebras \rm{(\cite{Lu2})}.

A {\bf matched pair of Lie algebras} (\cite{Majid,18}) consists of a pair of Lie algebras  $((\g,[\cdot,\cdot]_{\g}),(\h,[\cdot,\cdot]_{\h}))$, a  representation $\rho: \g\to\gl(\h)$ of $\g$ on $\h$ and a   representation $\mu: \h\to\gl(\g)$ of $\h$ on $\g$ such that
\begin{eqnarray}
\label{eq:mp1}\rho(x) [\xi,\eta]_{\h}&=&[\rho(x)\xi,\eta]_{\h}+[\xi,\rho(x) \eta]_{\h}+\rho\big(\mu(\eta)x\big)\xi-\rho\big(\mu(\xi)x\big) \eta,\\
\label{eq:mp2}\mu(\xi) [x, y]_{\g}&=&[\mu(\xi)x,y]_{\g}+[x,\mu(\xi) y]_{\g}+\mu\big(\rho(y) \xi\big)x-\mu\big(\rho(x)\xi\big)y,
\end{eqnarray}
 for all $x,y\in \g$ and $\xi,\eta\in \h$. We will denote a matched pair of Lie algebras by $(\g,\h;\rho,\mu)$, or simply by $(\g,\h)$.

The following alternative description of matched pairs of Lie algebras is well-known.
\begin{pro}\label{eqdefi}
Let $(\g,\h;\rho,\mu)$ be a matched pair of Lie algebras. Then there is a Lie algebra structure on the direct sum space $\g\oplus \h$ with the Lie bracket $[\cdot,\cdot]_{\bowtie }$ given by
\begin{eqnarray}\label{mathched-pair-Lie}
  [x+\xi,y+\eta]_{\bowtie }=\big([x,y]_\g+\mu(\xi)y-\mu(\eta) x\big)+\big([\xi,\eta]_\h+\rho(x)\eta-\rho(y) \xi),
\end{eqnarray}
where for all $x,y\in\g,\ \xi,\eta\in\h.$ Denote this Lie algebra by $\g\bowtie \h$.

Conversely, if $(\g\oplus \h,[\cdot,\cdot])$ is a Lie algebra such that $\g$ and $\h$ are Lie subalgebras, then  $(\g,\h;\rho,\mu)$ is a matched pair of Lie algebras, where the representations $\rho: \g\to\gl(\h)$  and $\mu: \h\to\gl(\g)$ are determined  by
\[[x,\xi] =\rho(x) \xi-\mu(\xi) x,\quad \forall x\in \g,\xi\in \h.\]
\end{pro}
\emptycomment{
\begin{thm}\label{thm:matched-Leibniz}
 Let $(\g,[\cdot,\cdot]_\g,P)$ be an averaging Lie algebra with the Leibniz algebra  $\g_P$.  Then $(\g_P,\g;\rho,\mu)$ is a {\bf matched pair of Leibniz algebras}, where $\rho:\g_P\to\gl(\g)$ and $\mu:\g\to \gl(\g_P)$ are given by
\begin{equation}\label{eq:mpr}
 \rho(\xi)(x)= B[x,\xi]_\g -[x,B\xi]_\g,\quad \mu(x)(\xi)=  [x,\xi]_\g,\quad \forall x\in\g, \xi\in\g_B.
\end{equation}
Moreover, the corresponding Lie algebra $\g_B\bowtie \g$ is exactly the Lie algebra $(\g\oplus \g,[\cdot,\cdot]_D)$ given in Proposition \ref{doublelie}.
\end{thm}

\begin{proof}
By Proposition \ref{doublelie},  it is obvious that both $\g_B$ and $\g$ are Lie subalgebras of the Lie algebra $(\g\oplus \g,[\cdot,\cdot]_D )$, and
$$
{[(\xi,0),(0,x)]_D} =(\rho(\xi)(x), -\mu(x)(\xi)).
$$ Thus,  $(\g_B,\g;\rho,\mu)$ is a matched pair of Lie algebras. The other conclusion is obvious.
\end{proof}}

Now we introduce the notion of matched pairs of Reynolds Lie algebras.

\begin{defi}\label{MPR}
A {\bf matched pair of Reynolds  Lie algebras} consists of a pair of Reynolds  Lie algebras  $((\g,[\cdot,\cdot]_{\g},R),(\h,[\cdot,\cdot]_{\h},\huaR))$, a  representation  $\rho: \g\to \gl(\h)$ of the Reynolds  Lie algebra $(\g,[\cdot,\cdot]_{\g},R)$ on $(\h,\huaR)$ and a   representation $\mu: \h\to \gl(\g)$ of the Reynolds  Lie algebra $(\h,[\cdot,\cdot]_{\h},\huaR)$ on $(\g,R)$ such that
  $(\g,\h; \rho,\mu)$ is a matched pair of Lie algebras.
\end{defi}

We will denote a matched pair of Reynolds  Lie algebras by $((\g,R),(\h,\huaR);\rho,\mu)$, or simply by $((\g,R),(\h,\huaR))$.

It is straightforward to obtain the following alternative characterization of matched pairs of Reynolds Lie algebras.
\begin{pro}
 Let $(\g,[\cdot,\cdot]_{\g},R)$ and $(\h,[\cdot,\cdot]_{\h},\huaR)$ be two Reynolds
 Lie algebras, $\rho: \g\to \gl(\h)$ be a  representation of the
 Lie algebra $\g$ on $\h$ and $\mu: \h\to \gl(\g)$ be a
 representation of the Lie algebra $\h$ on $\g$.
 Then $((\g,R),(\h,\huaR);\rho,\mu)$ is a matched pair of Reynolds
 Lie algebras if and only if $(\g,\h;\rho,\mu)$ is a matched pair of Lie algebras and the following equalities hold:
 \begin{eqnarray}
\label{Rey-1}\rho(Rx)(\huaR \xi)&=&\huaR\Big(\rho(x)(\huaR \xi)+\rho(Rx)(\xi)-\rho(Rx)(\huaR \xi)\Big),\\
\label{Rey-2}\mu(\huaR \xi)(Rx)&=&R\Big(\mu(\xi)(Rx)+\mu(\huaR \xi)(x)-\mu(\huaR \xi)(Rx)\Big),
\end{eqnarray}
for all $x\in \g$ and $\xi\in \h$.
\end{pro}

On the double of a matched pair of Lie algebras, there is a Lie algebra structure.
There is also a Reynolds  Lie algebra structure on the double of a matched pair of Reynolds Lie algebras.
\begin{pro}\label{Rey-on}
Let $(\g,[\cdot,\cdot]_{\g},R)$ and $(\h,[\cdot,\cdot]_{\h},\huaR)$ be Reynolds Lie algebras, $\rho:\g\to \gl(\h)$ and $\mu:\h\to \gl(\g)$ are linear maps. Then $((\g,R),(\h,\huaR);\rho,\mu)$ is a matched pair of Reynolds  Lie algebras if and only if $(\g,\h;\rho,\mu)$ is a matched pair of Lie algebras and
\[R\oplus \huaR:\g\oplus \h\to \g\oplus \h,\qquad x+\xi\mapsto Rx+\huaR\xi,\]
is a Reynolds  operator on the Lie algebra $\g\bowtie \h$.

\end{pro}
\begin{proof}
Let $((\g,R),(\h,\huaR);\rho,\mu)$ be a matched pair of Reynolds  Lie algebras. By Definition \ref{MPR}, $(\g,\h;\rho,\mu)$ is a matched pair of Lie algebras. By \eqref{Rey-1} and \eqref{Rey-2}, for all $x,y,\xi,\eta\in\g\oplus\h,$ we have
\begin{eqnarray*}
  &&[(R\oplus \huaR)(x+\xi),(R\oplus \huaR)(y+\eta)]_{\bowtie}-(R\oplus \huaR)[(R\oplus \huaR)(x+\xi),y+\eta]_{\bowtie}\\
 &&-(R\oplus \huaR)[x+\xi,(R\oplus \huaR)(y+\eta)]_{\bowtie}+(R\oplus \huaR)[(R\oplus \huaR)(x+\xi),(R\oplus \huaR)(y+\eta)]_{\bowtie}\\
  &=&[Rx,Ry]_{\g}+\mu(\huaR\xi)(Ry)-\mu(\huaR\eta)(Rx)+[\huaR\xi,\huaR\eta]_{\h}+\rho(Rx)(\huaR\eta)-\rho(Ry)(\huaR\xi)\\
  &&-R[Rx,y]_{\g}-R(\mu(\huaR\xi)y)+R(\mu(\eta)(Rx))-\huaR[\huaR\xi,\eta]_{\h}-\huaR(\rho(Rx)\eta)+\huaR(\rho(y)(\huaR\xi))\\
  &&-R[x,Ry]_{\g}-R(\mu(\xi)(Ry))+R(\mu(\huaR\eta)(x))-\huaR[\xi,\huaR\eta]_{\h}-\huaR(\rho(x)(\huaR\eta))+\huaR(\rho(Ry)(\xi))\\
  &&+R[Rx,Ry]_{\g}+R(\mu(\huaR\xi)(Ry))-R(\mu(\huaR\eta)(Rx))+\huaR[\huaR\xi,\huaR\eta]_{\h}+\huaR(\rho(Rx)(\huaR\eta))-\huaR(\rho(Ry)(\huaR\xi))\\
  &=&0.
\end{eqnarray*}
Therefore $R\oplus \huaR$ is a Reynolds  operator on the Lie algebra $\g\bowtie \h$.

The converse part can be proved similarly
by putting $\xi=y=0$ and $x=\eta=0$ respectively in the last equality.
\end{proof}

Let $((\g,R),(\h,\huaR);\rho,\mu)$ be a matched pair of Reynolds Lie algebras. Then there are three induced Lie algebras $\g_{R}, \h_{\huaR}$ and $(\g\bowtie \h)_{R\oplus \huaR}$ coming from the three Reynolds  operators $R:\g\to \g, \huaR:\h\to \h$ and $R\oplus \huaR:\g\oplus \h\to \g\oplus \h$, respectively. In the following we will show that a matched pair of Reynolds  Lie algebras gives rise to an induced matched pair of Reynolds Lie algebras.

\begin{thm}\label{demp}
Let $((\g,R),(\h,\huaR);\rho,\mu)$ be a matched pair of Reynolds Lie algebras, and $\g_R,~\h_\huaR$ be the induced Lie algebras. Then $(\g_{R},\h_{\huaR};\rho_{(R,\huaR)},\mu_{(R,\huaR)})$ is a matched pair of Lie algebras, where representations $\rho_{(R,\huaR)}$ and $\mu_{(R,\huaR)}$ are given by
\begin{eqnarray}
\label{eq:rep1}\rho_{(R,\huaR)}(x)\xi&=&\rho(x)(\huaR \xi)+\rho(Rx)\xi-\rho(Rx)(\huaR \xi),\\
\label{eq:rep2}\mu_{(R,\huaR)}(\xi)x&=&\mu(\xi)(Rx)+\mu(\huaR \xi)x-\mu(\huaR \xi)(Rx).
\end{eqnarray}
for all $x\in \g,\xi\in \h.$ Moreover, we have
\[\g_{R}\bowtie \h_{\huaR}=(\g\bowtie \h)_{R\oplus \huaR}\]
as Lie algebras.
\end{thm}
The matched pair $(\g_{R},\h_{\huaR};\rho_{(R,\huaR)},\mu_{(R,\huaR)})$ is called the {\bf  induced matched pair of Lie algebras}.

\begin{proof}
By Proposition \ref{Rey-on}, $(\g\oplus \h,[\cdot,\cdot]_{\bowtie},R\oplus \huaR)$ is a Reynolds  Lie algebra. By Proposition \ref{induce-Lie-algebra}, there is a induced Lie algebra on $\g\oplus \h$, denoted by $(\g\bowtie \h)_{R \oplus \huaR}$, which contains $\g_R$ and $\h_\huaR$ as Lie subalgebras.
Furthermore, for $x\in \g$ and $\xi\in \h$, we have
\begin{eqnarray*}
[x,\xi]_{(\g\bowtie \h)_{R\oplus \huaR}}
&=&[Rx,\xi]_{\bowtie}+[x,\huaR\xi]_{\bowtie}-[Rx,\huaR\xi]_{\bowtie}\\
&=&\rho(Rx)\xi-\mu(\xi)(Rx)+\rho(x)(\huaR \xi)-\mu(\huaR \xi)x-\rho(Rx)(\huaR \xi)+\mu(\huaR \xi)(Rx)\\
&=&\rho_{(R,\huaR)}(x)\xi-\mu_{(R,\huaR)}(\xi)x.
\end{eqnarray*}
Then by Proposition \ref{eqdefi}, $(\g_{R},\h_{\huaR};\rho_{(R,\huaR)},\mu_{(R,\huaR)})$ forms a matched pair of Lie algebras. Moreover, the induced Lie algebra on its double $\g_{R}\bowtie\h_{\huaR}$ coincides with $(\g\bowtie \h)_{R\oplus \huaR}$.
\end{proof}

\begin{cor}
Let $((\g,R),(\h,\huaR);\rho,\mu)$ be a matched pair of Reynolds Lie algebras. Then $((\g_{R},R)$,\\$(\h_{\huaR},\huaR);\rho_{(R,\huaR)},\mu_{(R,\huaR)})$ is a matched pair of Reynolds Lie algebras.
\end{cor}
\begin{proof}
By Proposition \ref{Rey-on} and Proposition \ref{induce-Lie-algebra},  $R\oplus \huaR$ is a Reynolds  operator on $(\g\bowtie \h)_{R\oplus \huaR}$. By Theorem \ref{demp}, $\g_{R}\bowtie \h_{\huaR}=(\g\bowtie \h)_{R\oplus \huaR}$. Therefore, $R\oplus \huaR$ is a Reynolds  operator on $\g_{R}\bowtie \h_{\huaR}$. By Proposition \ref{Rey-on},  $((\g_{R},R),(\h_{\huaR},\huaR);\rho_{(R,\huaR)},\mu_{(R,\huaR)})$ is a matched pair of Reynolds Lie algebras.
\end{proof}
\begin{ex}
Let $(\g,[\cdot,\cdot]_\g,R)$ and $(\h,[\cdot,\cdot]_\h,\huaR)$ be two Reynolds Lie algebras. Then $((\g,R),(\h,\huaR)$;\\$\rho=0,\mu=0)$ becomes a matched pair of Reynolds Lie algebras.
\end{ex}
\emptycomment{
\begin{rmk}\label{reason-1}
Let $(\g,[\cdot,\cdot]_{\g},R)$ and  $(\g^{*},[\cdot,\cdot]_{\g^{*}},R^{*})$ be two Reynolds Lie algebras. Then in general $((\g,R),(\g^*,R^*);\ad^*,\add^*)$ is not a matched pair of Reynolds Lie algebras. Let $\huaR=R^*$
in \eqref{Rey-1} and taking pairing with $y\in\g.$
The left hand side of \eqref{Rey-1} equals
\begin{eqnarray*}
\langle \ad^*_{Rx} (R^*\xi),y\rangle=-\langle\xi,R[Rx,y]_{\g}\rangle,
\end{eqnarray*}
and the right hand side of \eqref{Rey-1} equals to
\begin{eqnarray*}
\langle R^*(\ad^*_{x}(R^*\xi)+\ad^*_{Rx}\xi-\ad^*_{Rx}(R^*\xi)),y\rangle=-\langle\xi,R[x,Ry]_{\g}+[Rx,Ry]_{\g}-R[Rx,Ry]_{\g}\rangle.
\end{eqnarray*}
Observing that the relation is not satisfying \eqref{Rey-1}, thus $((\g,R),(\g^*,R^*);\ad^*,\add^*)$ is not a matched pair of Reynolds Lie algebras.

\end{rmk}}

\subsection{Manin triples of Reynolds  Lie algebras}

Recall that a {\bf Manin triple} of Lie algebras  is a triple $((\mathfrak{d},S),\g,\h)$, where $(\mathfrak{d},S)$ is a quadratic Lie algebra, $\g$ and $\h$ are Lie subalgebras of $\frkd$ such that
\begin{itemize}
\item [\rm (i)]  $\mathfrak{d}=\g\oplus \h$ as vector spaces;
\item [\rm (ii)] both $\g$ and $\h$ are isotropic with respect to the nondegenerate invariant symmetric bilinear form $S$.
\end{itemize}

It is well known that Manin triples of Lie algebras can be equivalently characterized by matched pairs of Lie algebras as follows.
\begin{pro}\label{pro:mp-mt}
  Let $(\g,[\cdot,\cdot]_{\g})$ and  $(\g^{*},[\cdot,\cdot]_{\g^{*}})$ be two Lie algebras. Then the following conditions are equivalent:
 \begin{itemize}
\item [\rm (i)]  $(\g,\g^{*};\ad^{*},\add^*)$ is a matched pair of Lie algebras, where $\add$ is the adjoint representation of $\g^*$.
\item [\rm (ii)]
 $((\g\oplus \g^{*},S),\g,\g^{*})$ is a Manin triple of  Lie algebras, where the invariant symmetric bilinear form $S$ on $\g\oplus \g^{*}$ is given by \begin{equation}\label{eq:sp}
  S(x+\xi,y+\eta)=\xi(y)+\eta(x).
\end{equation}
\end{itemize}

\end{pro}

Now we introduce the notion of Manin triples of Reynolds  Lie algebras using quadratic Reynolds  Lie algebras given in Definition \ref{defi:qua}.

\begin{defi}
A {\bf Manin triple of Reynolds  Lie algebras} is a triple   $((\huaG,\frkR,S), (\g,R), (\h,\huaR))$, where $(\huaG,\frkR,S)$ is a quadratic Reynolds Lie algebra, $(\g,R)$ and $(\h,\huaR)$ are Reynolds Lie algebras such that
\begin{itemize}
\item [\rm (i)]  $(\g,R)$ and $(\h,\huaR)$ are Reynolds Lie subalgebras, i.e. $\g$ and $\h$ are Lie subalgebras of $\huaG$ and $\frkR|_{\g}=R,~\frkR|_{\h}=\huaR$;
\item [\rm (ii)]  $\huaG=\g\oplus \h$ as the direct sum of vector spaces;
\item [\rm (iii)] both $\g$ and $\h$ are isotropic with respect to the nondegenerate symmetric invariant bilinear form $S$.
\end{itemize}
\end{defi}

Obviously, we have the following characterization of Manin triple of Reynolds  Lie algebras.
\begin{pro}\label{pro:mt}
 A triple   $((\huaG,\frkR,S), (\g,R), (\h,\huaR))$, where $(\huaG,\frkR,S)$ is a quadratic Reynolds  Lie algebra,  $(\g,R)$ and $(\h,\huaR)$ are Reynolds Lie algebras, is a Manin triple of Reynolds  Lie algebras if and only if $((\huaG,S), \g, \h)$ is a Manin triple of Lie algebras such that
 $$
 \frkR|_{\g}=R,~\frkR|_{\h}=\huaR.
 $$
\end{pro}

Let $(\g,[\cdot,\cdot]_{\g})$ be a Lie algebra and suppose that the dual space $\g^*$
is equipped with a Lie algebra brackets $(\g^*,[\cdot,\cdot]_{\g^*})$.
Define maps $\ad^*:\g\rightarrow \gl(\g^*)$ and $\add^*:\g^*\rightarrow \gl(\g)$ by
\[\langle \ad_x^*\xi,y\rangle =-\langle \xi,[x,y]_{\g}\rangle,\quad\langle \add_{\xi}^*x,\eta\rangle=-\langle x,[\xi,\eta]_{\g^*}\rangle.\]
\begin{thm}\label{matched-manin-equivalent}
Let $(\g,[\cdot,\cdot]_{\g})$ and $(\g^*,[\cdot,\cdot]_{\g^*})$ be two
 Lie algebra structures on a vector space $\g$ and it's dual $\g^*$.
 Suppose that $R:\g\rightarrow \g$ is a linear map such that $(\g,[\cdot,\cdot]_{\g},R)$
 and  $(\g^{*},[\cdot,\cdot]_{\g^{*}},-R^{*})$ are two Reynolds Lie algebras.
 Then the following condition are equivalent.
 \begin{itemize}
\item [\rm (i)]  $((\g,R),(\g^{*},-R^*);\ad^{*},\add^*)$ is a matched pair of Reynolds Lie algebras.
\item [\rm (ii)]
 $((\g\oplus \g^{*},R-R^*,S),(\g,R),(\g^{*},-R^*))$ is a Manin triple of Reynolds Lie algebras, where the invariant symmetric bilinear form $S$ on $\g\oplus \g^{*}$ is given by \eqref{eq:sp}, and $$R-R^*:\g\oplus \g^{*}\rightarrow\g\oplus \g^{*}, x+\xi \mapsto Rx-R^*\xi.$$
\end{itemize}
\end{thm}

\begin{proof}

  Let $((\g,R),(\g^{*},-R^*);\ad^{*},\add^*)$ be a matched pair of Reynolds Lie algebras. By Proposition \ref{Rey-on}, $(\g,\g^*;\ad^{*},\add^*)$ is a
  matched pair of Lie algebras, and
  $(\g\oplus \g^{*},[\cdot,\cdot]_{\bowtie},R-R^*)$ is a Reynolds Lie algebra,
  where $[\cdot,\cdot]_{\bowtie}$ is given by \eqref{mathched-pair-Lie}. By Proposition \ref{pro:mp-mt},  $((\g\oplus \g^{*},S),\g,\g^{*})$ is a Manin triple of Lie algebras. Then by Proposition \ref{pro:mt}, we
  only need to prove that $S$ satisfies the invariant conditions \eqref{Invariant} and \eqref{Rey-manin}. The fact that $S$ satisfies the invariant condition \eqref{Invariant} follows from Proposition \ref{pro:mp-mt} directly.
  For all $x,y\in \g$ and $\xi,\eta\in \g^*,$
   we have
  \begin{eqnarray*}
    &&S((R-R^*)(x+\xi),y+\eta)+S(x+\xi,(R-R^*)(y+\eta))\\
    &=&S(Rx-R^*\xi,y+\eta)+S(x+\xi,Ry-R^*\eta)\\
    &=&-\langle R^*\xi,y\rangle+\langle Rx,\eta\rangle+\langle \xi,Ry\rangle-\langle x,R^*\eta\rangle\\
    &=&0.
  \end{eqnarray*}
Thus, $S$ satisfies the invariant condition \eqref{Rey-manin}. Therefore, $((\g\oplus \g^{*},R-R^*,S),(\g,R),(\g^{*},-R^*))$ is a Manin triple of Reynolds Lie algebras.

On the other hand, if $((\g\oplus \g^{*},R-R^*,S),(\g,R),(\g^{*},-R^*))$ is a Manin triple of Reynolds Lie algebras, where the invariant symmetric bilinear form $S$ on $\g\oplus \g^{*}$ is given by \eqref{eq:sp}, then $((\g\oplus \g^{*},S),\g,\g^{*})$ is a Manin triple of Lie algebras. By Proposition \ref{pro:mp-mt}, $(\g,\g^{*};\ad^{*},\add^*)$ is a matched pair of  Lie algebras. Since $R-R^*$ is a Reynolds  operator on $\g\bowtie\g^*$,
 $((\g,R),(\g^{*},-R^*);\ad^{*},\add^*)$ is a matched pair of Reynolds Lie algebras by Proposition \ref{Rey-on}.
\end{proof}

\section{Reynolds Lie bialgebras and Reynolds operators on quadratic Rota-Baxter Lie algebras}\label{sec:bia}

\subsection{Reynolds operators on Lie bialgebras}

In this subsection, we introduce the notion of Reynolds Lie bialgebras, and show that Reynolds Lie bialgebras and certain matched pairs of Reynolds Lie algebras are equivalent.
First, we recall the notion of Lie coalgebras and Lie bialgebras.

\begin{defi}
A vector space $\g$ with a linear map $\Delta:\g\rightarrow\g\ot \g$ is called a {\bf Lie coalgebra} if $\Delta$ is {\bf coantisymmetric}, in the sense that $\Delta=-\tau \circ\Delta$ for the flip map $\tau:\g\otimes\g\rightarrow \g\otimes \g$, and satisfies the {\bf co-Jacobian identity}:
\begin{equation}
(\Id +\epsilon+\epsilon^2)(\Id \otimes \Delta)\Delta =0,
\end{equation}
where $\epsilon(x\otimes y\otimes z):=z\otimes x\otimes y$ for all $x, y, z\in \g$.
\end{defi}

\begin{defi}
A {\bf Lie bialgebra} is a Lie algebra $(\g,[\cdot,\cdot]_{\g})$ equipped with a linear map  $\Delta:\g\rightarrow\g\otimes\g$ such that $\Delta$ defines a Lie coalgebra and the following condition is satisfied:
  \begin{eqnarray}
   \Delta[x,y]_{\g}=(\ad_x\otimes \Id+\Id\otimes \ad_x)\Delta(y)-(\ad_y\otimes \Id+\Id\otimes \ad_y)\Delta(x),\quad \forall x,y\in \g.
  \end{eqnarray}
\end{defi}
\begin{rmk}
In fact, $(\g,\Delta)$ is a Lie coalgebra if and only if $(\g^*,[\cdot,\cdot]_{\g^*})$ is a Lie algebra, where $[\cdot,\cdot]_{\g^*}:\g^*\otimes\g^*\rightarrow\g^*$ is the linear dual of $\Delta,$
that is
$$\langle \Delta(x),\xi\otimes \eta\rangle:=\langle x,[\xi,\eta]_{\g^*}\rangle. \quad \forall x\in \g, \xi,\eta\in\g^*.$$
\end{rmk}

A Lie bialgebra is denoted by $(\g,\g^*)$ or $(\g,\Delta)$.
For such a Lie bialgebra, there is a Lie algebra structure $[\cdot,\cdot]_{\bowtie}$ on the double $\frkd:=\g\oplus \g^*$ such that $\g$ and $\g^*$ are Lie subalgebras and
\[[x,\xi]_{\bowtie }=-\add_{\xi}^*x+\ad_x^*\xi,\quad \forall x\in \g,\xi\in \g^*.\]

Denote this Lie algebra by $\g\bowtie \g^*$. We refer the readers to \rm{(\cite{Lu2})} for the equivalent description of Lie bialgebras by using Manin triples of Lie algebras. One important observation is that if $(\g,\g^*)$ is a Lie bialgebra, then $(\g^*,\g)$ is also a Lie bialgebra.



\begin{defi}
A {\bf Reynolds operator} on a Lie bialgebra $(\g,\g^*)$ is a linear map $R:\g\to \g$ such that
\begin{itemize}
\item[\rm (i)] $R$ is a Reynolds operator on the Lie algebra $\g$;
\item[\rm (ii)]  ${-R}^*$ is a Reynolds operator on the Lie algebra $\g^*$.
\end{itemize}
A Lie bialgebra with a Reynolds operator is called a {\bf Reynolds Lie bialgebra}.
\end{defi}
We denote a Reynolds Lie bialgebra by $(\g,\g^*,R)$. The following result is straightforward.
\begin{cor}
If $R$ is a Reynolds operator on the Lie bialgebra $(\g,\g^*)$, then ${-R}^*$ is a Reynolds operator on the Lie bialgebra $(\g^*,\g)$.
\end{cor}

\begin{thm}\label{bialgebra-matched}
 Let $(\g,[\cdot,\cdot]_{\g},R)$ and  $(\g^{*},[\cdot,\cdot]_{\g^{*}},{-R}^{*})$ be two Reynolds Lie algebras. Then $(\g,\g^{*},R)$ is a Reynolds Lie bialgebra if and only if $((\g,R),(\g^*,{-R}^*);\ad^*,\add^*)$ is a matched pair of Reynolds Lie algebras.
\end{thm}

\begin{proof}
It is obvious that if $((\g,R),(\g^*,{-R}^{*});\ad^*,\add^*)$ is a matched pair of Reynolds Lie algebras, then  $(\g,\g^{*},R)$ is a Reynolds Lie bialgebra.

Conversely, let $(\g,\g^*)$ be a Lie bialgebra, $R:\g\to \g$ and ${-R}^{*}:\g^*\to \g^*$ be Reynolds operators. By Corollary \ref{thm:coadjoint},  $(\g^*;{-R}^{*},\ad^*)$ is a representation of the Reynolds Lie algebra $(\g,R)$, and $(\g;R,\add^*)$ is a representation of the Reynolds Lie algebra $(\g^*,{-R}^{*})$. Moreover, it is well-known that  $(\g,\g^* )$ is a  Lie bialgebra if and only if $(\g,\g^*;\ad^*,\add^*)$ is a matched pair of  Lie algebras.
Therefore,
 $((\g,R),(\g^*,{-R}^{*});\ad^*,\add^*)$ is a matched pair of Reynolds Lie algebras.
\end{proof}

\begin{rmk}
  Let $(\g,\g^*)$ be a Lie bialgebra, $R:\g\rightarrow\g$ and $R^*:\g^*\rightarrow\g^*$ are
  Reynolds operators. Then $((\g,R),(\g^*,R^*);\ad^*,\add^*)$ is not necessarily a matched pair of Reynolds Lie algebras.
  Let $\huaR=R^*$ in \eqref{Rey-1} and take the pairing with $y\in\g$. The right hand side of \eqref{Rey-1} equals
\begin{eqnarray*}
\langle R^*(\ad^*_{x}(R^*\xi)+\ad_{Rx}^*\xi-\ad_{Rx}^*(R^*\xi)),y\rangle=-\langle\xi,R[x,Ry]_{\g}+[Rx,Ry]_{\g}-R[Rx,Ry]_{\g}\rangle,
\end{eqnarray*}
and the left hand side of \eqref{Rey-1} equals
\begin{eqnarray*}
\langle\ad_{Rx}^*(R^*\xi),y\rangle=-\langle\xi,R[Rx,y]_{\g}\rangle.
\end{eqnarray*}
This means that the condition \eqref{Rey-1} is not necessarily fulfilled, thus $((\g,R),(\g^*,R^*);\ad^*,\add^*)$ is not a matched pair of Reynolds Lie algebras.
\end{rmk}

\begin{cor}\label{cor:dd}
A triple $(\g,\g^{*},R)$ is a Reynolds Lie bialgebra if and only if $(\frkd,\frkR)$ is a Reynolds Lie algebra, where $\frkR:\g\oplus\g^*\rightarrow\g\oplus\g^*$ is the linear map defined by

\begin{eqnarray}\label{R-defined}
\frkR(x+\xi)=Rx-R^*\xi,\quad \forall x\in\g, \xi\in \g^*.
\end{eqnarray}
\end{cor}
\begin{proof}
Let $(\g,\g^*,R)$ be a Reynolds Lie bialgebra. By Theorem \ref{bialgebra-matched}, $((\g,R),(\g^*,{-R}^{*});\ad^*,\add^*)$ is a matched pair of Reynolds Lie algebras. By Proposition \ref{Rey-on}, the linear map $\frkR$ is a Reynolds operator on the Drinfeld double  $\mathfrak{d}:=\g\bowtie \g^*$. The inverse statement is obvious.
\end{proof}

Let $(\g,\g^*)$ be an arbitrary Lie bialgebra. Let $\{e_1,e_2,\cdots,e_n\}$ be a basis of $\g$ and $\{\xi_1,\xi_2,\cdots,\xi_n\}$ be the dual basis of $\g^*$.
Then $r=\sum_{i}e_i\otimes \xi_i\in \g\otimes\g^*\subset\mathfrak{d}\otimes \mathfrak{d} $ induces a Lie algebra structure on $\mathfrak{d}^*$, where the Lie bracket $[\cdot,\cdot]_{{r}}$ on $\mathfrak{d}^*$ is defined by
\begin{eqnarray*}
[\xi+x,\eta+y]_{{r}}=-[\xi,\eta]_{\g^*}+[x,y]_{\g},\quad \forall x,y \in \g,\xi,\eta\in \g^*.
\end{eqnarray*}
Denote this Lie algebra by $\frkd^*_r$.
 Moreover,   $(\mathfrak{d},\mathfrak{d}_{r}^*)$ is a quasi-triangular Lie bialgebra (\cite{RS}).

The following result tells us that the Drinfeld double of a Reynolds Lie bialgebra is still a Reynolds Lie bialgebra.

\begin{pro}\label{Double-r-bia}
Let $(\g,\g^*,R)$ be a Reynolds Lie bialgebra. Then $(\mathfrak{d},\mathfrak{d}_{r}^*,\frkR)$ is a Reynolds Lie bialgebra,
where $\frkR: \g\oplus \g^*\to \g\oplus \g^*$ is the linear map defined by \eqref{R-defined}.
\end{pro}

\begin{proof}
By Corollary \ref{cor:dd}, $\frkR$
is a Reynolds operator on the Drinfeld double  $\mathfrak{d}:=\g\bowtie \g^*$. Thus we only need to show that
 $-\frkR^*:\mathfrak{d}_{r}^*\rightarrow\mathfrak{d}_{r}^*$, where $-\frkR^*(\xi+x)=-R^*\xi+Rx$ is a Reynolds operator on $\mathfrak{d}_{r}^*$. This follows from the fact that $R$ and $-R^{*}$ are Reynolds operators on $\g$ and $\g^*$ respectively.
Consequently, $(\mathfrak{d},\mathfrak{d}_{r}^*,\frkR)$ is a Reynolds Lie bialgebra.
\end{proof}

\subsection{Reynolds operators on quadratic Rota-Baxter Lie algebras}
In this subsection, first we recall the relation between factorizable Lie bialgebras
and quadratic Rota-Baxter Lie algebras established in \rm{(\cite{Lang-Sheng})}. Then we introduce the notion of Reynolds operators on quadratic Rota-Baxter Lie algebras and show that they can give rise to Reynolds Lie bialgebras.
\begin{defi}
  A Lie bialgebra $(\g,\Delta)$ is called {\bf coboundary} if $\Delta$ is a $1$-coboundary of the Lie algebra $\g$ with coefficients in the representation $\ad\otimes \Id+\Id\otimes\ad$ on the tensor space $\g\otimes \g$,
  that is, there exists $r\in \g\otimes\g$ such that
  \begin{eqnarray}\label{r-coboundary}
    \Delta(x)=(\ad_{x}\otimes\Id+\Id \otimes\ad_{x})r, \quad \forall x\in \g.
  \end{eqnarray}
\end{defi}

Let $r=\sum_i x_i\otimes y_i\in \g\otimes \g.$ Let $U(\g)$ be the universal enveloping algebra of the Lie algebra $\g.$ We introduce $r_{12},r_{13},r_{23}$ in $U(\g)\otimes U(\g)\otimes U(\g)$ as follows:
$$r_{12}=\sum_i x_i\otimes y_i\otimes 1,\quad r_{23}=\sum_i 1\otimes x_i\otimes y_i,\quad r_{13}=\sum_i x_i\otimes 1\otimes y_i,$$ and the bracket $[r_{12},r_{13}]$ is defined by
\[\left[r_{12},r_{13}\right]=\left[\sum_i x_i\otimes y_i\otimes 1,\sum_j x_j\otimes 1\otimes y_j\right]=\sum_{i,j}[x_i,x_j]_\g\otimes y_i\otimes y_j, \]
and similarly for $[r_{13},r_{23}]$ and $[r_{12},r_{23}]$.

 \begin{thm}\label{r-bialgebra} {\rm(\cite{CP})} Let $(\g,[\cdot,\cdot]_{\g})$ be a Lie algebra and $r\in \g\otimes \g$. Then
the map $\Delta$ defined by \eqref{r-coboundary} induces a Lie algebra structure on $\g^*$ such that $(\g,\Delta)$ is a Lie bialgebra if and only if for all $x\in \g,$ the following conditions are satisfied:
 \begin{eqnarray}
\label{eq:invr}\big(\ad_x\otimes \Id+\Id\otimes\ad_x\big)(r+\sigma(r))&=&0,\\
 \big(\ad_x\otimes \Id\otimes \Id+\Id\otimes\ad_x\otimes\Id+\Id\otimes \Id\otimes \ad_x\big)[[r,r]]&=&0,
 \end{eqnarray}
where $\sigma:\g\otimes \g\rightarrow \g\otimes \g$ is the flip operator which interchanges the components in $\g\otimes \g$ and $[[r,r]]$ is defined by
\begin{eqnarray}\label{r-matrix}
[[r,r]]=[r_{12},r_{13}]+[r_{13},r_{23}]+[r_{12},r_{23}].
\end{eqnarray}
We denote this {\bf coboundary Lie bialgebra} by $(\g,\g^*_{r}).$
\end{thm}
In particular, the equation
\begin{eqnarray*}
[[r,r]]=0
\end{eqnarray*}
is called the {\bf classical Yang-Baxter equation} in a Lie algebra and a solution of the classical Yang-Baxter equation in a Lie algebra is also called a {\bf
        classical $r$-matrix}.

        Any $r\in \g\otimes \g$ induces a linear operator $r_+: \g^*\to \g$ defined by $r_+( \xi)= r(\xi,\cdot)$ for all $\xi\in\g^*.$
  Explicitly, writing $r=\sum_{i}x_i\otimes y_i$, the corresponding map $r_+$ is
\begin{equation}\label{sharp-condition}
    r_+(\xi)=\sum_{i}\langle \xi, x_i \rangle y_i,
    ~~\forall \xi\in \g^*.
\end{equation}

 Note that $r$ is skew-symmetric if and only if
\begin{equation}
    \langle r_+(\xi),\eta\rangle+\langle \xi,r_+(\eta)\rangle=0,\;\;\forall
   \xi,\eta\in \g^*.\label{eq:skew-symmetry}
\end{equation}

 Define $$r_-:=-r_+^*.$$
 Let us introduce a bracket $[\cdot,\cdot]_r$ on $\g^*$:
\begin{eqnarray}\label{r-bracket}
[\xi,\eta]_r=\ad_{r_+ \xi}^* \eta-\ad^*_{r_- \eta} \xi,\quad \forall \xi,\eta\in \g^*.
\end{eqnarray}
If $r$ satisfies the $\ad$-invariant condition
 \eqref{eq:invr} and the classical Yang-Baxter equation,  then $(\g^*,[\cdot,\cdot]_r)$ is a Lie algebra, which is denoted by $\g_r^*$. Moreover, $(\g,\g^*_r)$ constitutes a Lie bialgebra, which is called   a {\bf quasi-triangular Lie bialgebra}. If $r$ is skew-symmetric, it is called a {\bf triangular Lie bialgebra}.

Denote by $I$ the operator
\begin{eqnarray}\label{I}
I=r_+-r_-: \g^*\to \g.
\end{eqnarray}
Note that $I^*=I$, and the $\ad$-invariant condition \eqref{eq:invr} is equivalent to
\begin{equation}\label{eq:invI}
  I\circ \ad_x^*=\ad_x\circ I,\quad \forall x\in \g.
\end{equation} Actually $\frac{1}{2}I$ is the symmetric part of $r$.
If $r$ is skew-symmetric, then $I=0$. Factorizable Lie bialgebras  are however
concerned with the opposite case that $I$ is nondegenerate; see \rm{(\cite{RS})}.
\begin{defi}
A quasi-triangular Lie bialgebra $(\g, \g^*_r)$ defined by $r\in \g\otimes \g$ is called {\bf factorizable} if the linear map $I: \g^*\to \g$ defined by \eqref{I} is a linear isomorphism of vector spaces.
\end{defi}

\begin{defi}
Let $(\g,[\cdot,\cdot]_\g)$ be a Lie algebra.
A linear operator $B:\g\longrightarrow \g$ is called a {\bf Rota-Baxter operator of weight $\lambda$} if
\begin{equation*}
 [Bx,By]_\g=B\Big([Bx,y]_\g+ [x,By]_\g+\lambda[x,y]_{\g}\Big), \quad \forall x, y \in \g.
\end{equation*}
Moreover, a Lie algebra $(\g,[\cdot,\cdot]_\g)$ with a Rota-Baxter operator $B$ is
called a {\bf Rota-Baxter Lie algebra of weight $\lambda$}. We denote it by $(\g,[\cdot,\cdot]_\g,B)$.
\end{defi}

Let $(\g, [\cdot,\cdot]_\g)$ be a Lie algebra  and $B:\g\lon\g$  a Rota-Baxter operator of weight $\lambda$ on $\g$. Then there is a new Lie bracket $[\cdot,\cdot]_B$ on $\g$ defined by
$$
[x,y]_B=[Bx,y]_\g+[x,By]_\g+\lambda[x,y]_\g.
$$
The Lie algebra $(\g,[\cdot,\cdot]_B)$ is called the {\bf descendent Lie algebra}, and denoted by $\g_B$. It is obvious that $B$ is a Lie algebra homomorphism from $\g_B$ to $\g$:
$$
B([x,y]_B)=[Bx,By]_\g.
$$

\begin{pro}\label{rota-baxterLie-ave}
  Let $(\g,[\cdot,\cdot]_\g,B)$ be a Rota-Baxter Lie algebra of weight $\lambda$ and $R:\g\rightarrow\g$  a Reynolds operator on $\g$. If $R\circ B=B\circ R, $ then $R$ is also a Reynolds operator on the
  descendent Lie algebra $(\g,[\cdot,\cdot]_B)$.
\end{pro}
\begin{proof}
Let $(\g,[\cdot,\cdot]_\g,B)$ be a Rota-Baxter Lie algebra of weight $\lambda$. For all $x,y \in \g$, we have
\begin{eqnarray*}
[Rx,Ry]_{B}&=&[B(Rx),Ry]_{\g}+[Rx,B(Ry)]_{\g}+\lambda[Rx,Ry]_{\g}\\
&=&[R(Bx),Ry]_{\g}+[Rx,R(By)]_{\g}+\lambda[Rx,Ry]_{\g}\\
&=&R[R(Bx),y]_{\g}+R[Bx,Ry]_{\g}-R[R(Bx),Ry]_{\g}\\
&&+R[x,R(By)]_{\g}+R[Rx,By]_{\g}-R[Rx,R(By)]_{\g}\\
&&+\lambda R[x,Ry]_{\g}+\lambda R[Rx,y]_{\g}-\lambda R[Rx,Ry]_{\g}\\
&=&R[Rx,y]_{B}+R[x,Ry]_{B}-R[Rx,Ry]_{B},
\end{eqnarray*}
which implies that $R$ is a Reynolds operator on the
  descendent Lie algebra $(\g,[\cdot,\cdot]_B).$
 \end{proof}

\begin{defi}\rm{(\cite{Lang-Sheng})}
Let $(\g,[\cdot,\cdot]_\g,B)$ be a  Rota-Baxter Lie algebra of weight $\lambda$, and $S\in \otimes ^2\g^*$ a nondegenerate symmetric bilinear form.   The triple  $(\g,B,S)$ is called a {\bf quadratic Rota-Baxter Lie algebra of weight $\lambda$} if $(\g,S)$ is a quadratic Lie algebra and the following compatibility condition holds:
\begin{eqnarray}
\label{RBmanin}
S( x, {B} y)+S({B}x, y)+\lambda S(x,y)&=&0,\quad \forall x,y\in \g.
\end{eqnarray}

 It is known that quadratic Rota-Baxter Lie algebras of  weight 1 (or any nonzero weight) are one-to-one correspondence with factorizable Lie bialgebras.

\end{defi}

\begin{thm}\label{converseW0}\rm{(\cite{MG,Lang-Sheng,STS})}
    Let $(\g, B,S)$ be a quadratic Rota-Baxter Lie algebra of weight $\lambda$ and $\huaI_S:\g^*\to\g$ the induced linear isomorphism given by $\langle \huaI_S^{-1} x,y\rangle:=S(x,y)$. Then  $r^{B,S}\in \g\otimes \g$ determined by
\[r^{B,S}_+:=B\circ \huaI_S:\g^*\to \g, \quad r^{B,S}_+(\xi)=r^{B,S}(\xi,\cdot),\quad \forall\xi\in \g^*,\]
satisfies the classical Yang-Baxter equation, and gives rise to a triangular (if $\lambda=0$) or a factorizable (if $\lambda\neq 0$) Lie bialgebra $(\g,\g_{r^{B,S}}^*)$ with a comultiplication $\Delta_{r^{B,S}}$ defined as $\Delta_{r^{B,S}}(x)=(\ad_x\otimes \Id +\Id\otimes \ad_x)r^{B,S}$. Moreover, for all $x,y\in \g$
$$
[I_S^{-1}x,I_S^{-1}y]_{\g^*_{r_{B,S}}}=I_S^{-1}([x,y]_B).
$$
\end{thm}


Now we introduce the notion of a Reynolds  operator on a quadratic Rota-Baxter Lie algebra, and show that  a Reynolds  operator on a quadratic Rota-Baxter Lie algebra naturally induces a Reynolds Lie bialgebra.



\begin{defi}
  Let  $(\g, B,S)$ be a quadratic Rota-Baxter Lie algebra of weight $\lambda$. A linear map $R:\g\to\g$ is called a Reynolds operator on  $(\g, B,S)$, if $B\circ R^*=-R\circ B$.
\end{defi}

\begin{rmk}
    If $(\g,B,S)$ is a quadratic Rota-Baxter Lie algebra of weight $\lambda$, then from condition  $B\circ R^*=-R\circ B$ it follows that $\lambda(R+R^*)=0$. That is, if $\lambda\neq 0$, then any Reynolds operator on $(\g,B,S)$ is skew-symmetric.
\end{rmk}
\begin{pro}\label{inv-Der}
    Let $(\g,B,S)$ be a quadratic Rota-Baxter Lie algebra and $d\in\Der(\g)$ be a derivation of $\g$ such that $\Id-d$ is an invertible map.  Then the Reynolds operator $R=(\Id-d)^{-1}$ is a Reynolds operator on $(\g,B,S)$ if and only if
   \begin{equation}\label{der-Rey-RB}
    d\circ B+B\circ d^*=2B.
    \end{equation}
    In particular, if the derivation $d$ is an inner derivation, then the Reynolds operator $R$ is a Reynolds operator on $(\g,B,S)$ if and only if
    \begin{equation}\label{der-inn}
        [B,d]=-2B.
    \end{equation}
\end{pro}

\begin{proof}
If $R=(\Id-d)^{-1}$, then the equality $B\circ R^*=-R\circ B$ is equivalent to the following condition
$$
(\Id-d)\circ B=-B\circ(\Id-d)^*
$$
that is equivalent to \eqref{der-Rey-RB}.
If $d$ is the inner derivation, that is, $d=\ad_a$ for some $a\in\g$, then $d^*=-d$. Indeed, for any $x,y\in \g$, by \eqref{Invariant}, we have
  $$ S(d^*(x),y)=S(x,d(y))=S(x,[a,y])=-S([a,x],y)=S(-d(x),y). $$
Thus, if $d$ is the inner derivation, then $R=(\Id-d)^{-1}$ is a Reynolds operator on $(\g,B,S)$  if and only if the equality \eqref{der-inn} holds.
\end{proof}

\begin{thm}
    Let  $(\g, B,S)$ be a quadratic Rota-Baxter Lie algebra of weight $\lambda$ and $R:\g\rightarrow \g$ be a Reynolds operator on $(\g, B,S)$. Then $-R^*$ is a Reynolds operator on the descendent Lie algebra $\g_B.$
\end{thm}

\begin{proof}
Since $(\g, B,S)$ is a quadratic Rota-Baxter Lie algebra of weight $\lambda$, then the adjoint map
$B^*=-B-\lambda\mathrm{id}$ is also a Rota-Baxter operator of weight $\lambda$.
For any $x,y\in\g$ we have
\begin{eqnarray*}
[-R^*x,-R^*y]_B&=&[B(R^*x),R^*y]_{\g}+[R^*x,B(R^*y)]_{\g}+\lambda[R^*x,R^*y]_{\g}\\
&=&[B(R^*x),R^*y]_{\g}-[R^*x,B^*(R^*y)]_{\g}.
\end{eqnarray*}
 Recall that if $\lambda\neq 0$, then $R^*=-R$. It means that for any weight $\lambda$, the map $B^*$ also satisfies the condition $R\circ B^*=-B^*\circ R^*$. Consider $z\in\g$. Then
\begin{eqnarray*}
    S([B(R^*x),R^*y]_{\g},z)&=&S(-[R(Bx),R^*y]_{\g},z)\\
    &=&-S(y,R([z,R(Bx)]_{\g}))\\
    &=&-S(y,[Rz,R(Bx)]_{\g}+R([Rz,R(Bx)]_{\g})-R([Rz,Bx]_{\g}))\\
    &=&-S(R^*([R(Bx),y]_{\g})+R^*([R(Bx),R^*y]_{\g})-R^*([Bx,R^*y]_{\g}),z)\\
   &=&S(R^*([(B\circ R^*)x,y]_{\g})+R^*([(B\circ R^*)x,R^*y]_{\g})+R^*([Bx,R^*y]_{\g},z).
\end{eqnarray*}
Therefore,
$$[B(R^*x),R^*y]_{\g}=R^*([(B\circ R^*)x,y]_{\g}+[(B\circ R^*)x,R^*y]_{\g}+[Bx,R^*y]_{\g}).$$
Similarly,
\begin{eqnarray*}
-[R^*x,B^*(R^*y)]_{\g}&=&[B^*(R^*y),R^*x]_{\g}\\
&=&-R^*([(B^*\circ R^*)y),x]_{\g})+R^*([(B^*\circ R^*)y),R^*x]_{\g})+R^*([B^*y,R^*x]_{\g})\\
&=&R^*([x,(B\circ R^*)y]_{\g}+[R^*x,(B\circ R^*)y]_{\g}+[R^*x,By]_{\g}\\
&&+\lambda [x,R^*y]_{\g}+\lambda [R^*x,R^*y]_{\g}+\lambda [R^*x,y]_{\g}).
\end{eqnarray*}
Therefore,
\begin{eqnarray*}
    [-R^*x,-R^*y]_B&=&[B(R^*x),R^*y]_{\g}-[R^*x,B^*(R^*y)]_{\g}\\
    &=&R^*([(B\circ R^*)x,y]_{\g}+[(B\circ R^*)x,R^*y]_{\g}+[Bx,R^*y]_{\g})\\
    &+&R^*([x,(B\circ R^*)y]_{\g}+[R^*x,(B\circ R^*)y]_{\g}+[R^*x,By]_{\g})\\
    &+&R^*(\lambda [x,R^*y]_{\g}+\lambda [R^*x,R^*y]_{\g}+\lambda [R^*x,y]_{\g})\\
    &=&R^*([R^*x,y]_B+[x,R^*y]_B+[R^*x,R^*y]_B)\\
    &=&-R^*([-R^*x,y]_B+[x,-R^*y]_B-[-R^*x,-R^*y]_B).
\end{eqnarray*}
Thus, $-R^*$ is a Reynolds operator on $\g_B$.
\end{proof}

\begin{thm}\label{thm:FL}
Let  $(\g, B,S)$ be a quadratic Rota-Baxter Lie algebra of weight $\lambda$. Let $R:\g\to\g$ be a Reynolds operator on $(\g,B,S)$. Then $(\g,\g_{r^{B,S}}^*,R)$ is a Reynolds Lie bialgebra.
\end{thm}

\begin{proof}
Let  $(\g, B,S)$ be a quadratic Rota-Baxter Lie algebra of weight $\lambda$. By Theorem \ref{converseW0}, the pair $(\g,\g_{r^{B,S}}^*)$ is a Lie bialgebra (triangular or factorizable). Since $R$ is a Reynolds operator on the Lie algebra $\g$, we only need to show that $-R^*$ is a Reynolds operator on the Lie algebra $\g_{r^{B,S}}^*$. To avoid an ambiguity, in what follows we will mean by $R^*$ the dual map defined as
$\langle R^*f,x\rangle=\langle f,Rx\rangle$,
while by $R^{*,S}$ we will mean the adjoint map defined as $S(Ra,b)=S(a,R^{*,S}b)$. Obviously, $R^*\circ I_S^{-1}=I_S^{-1}\circ R^{*,S}$. For any $f=I_S^{-1}(x),g=I_S^{-1}(y)\in \g^*$ and $c\in\g$, we have
\begin{eqnarray*}
\left[(-R^*)f,(-R^*)g\right]_{g^*_{r^{B,S}}} &=& \left[I_S^{-1}(-R^{*,S}a),I_S^{-1}(-R^{*,S}b)\right ]\\ &=&I_S^{-1}\left(\left[-R^{*,S}x,-R^{*,S}y\right]_B\right)\\
&=&I_S^{-1}\left( -R^{*,S}([-R^{*,S}x,y]_B+[x,-R^{*,S}y]_B-[-R^{*,S}x,-R^{*,S}y]_B) \right)\\
&=&-R^*([-R^*f,g]_{g^*_{r^{B,S}}}+[f,-R^*g]_{g^*_{r^{B,S}}}-[-R^*f,-R^*g]_{g^*_{r^{B,S}}}).
\end{eqnarray*}
Therefore, $(\g,\g_{r^{B,S}}^*,R)$ is a Reynolds Lie bialgebra.
\end{proof}

\begin{ex}
Let $sl(2,\mathbb C)$ be the Lie algebra from Example \eqref{dre-Rey-ex}.
Let $\{h^*, x^*, y^*\}$  be the dual basis of the dual vector space $\g^*.$
Moreover, $(\g^*,[\cdot,\cdot]_{\g^*})$ is a Lie algebra, where the Lie bracket $[\cdot,\cdot]_{\g^*}$ are given by
\begin{eqnarray*}
[h^*,x^*]_{\g^*}=\frac{1}{4}x^*,\quad [h^*,y^*]_{\g^*}=\frac{1}{4}y^*,\quad [x^*,y^*]_{\g^*}=0,
\end{eqnarray*}
and $(\g,\g^*)$ is a Lie bialgebra (\cite{KM}). Let $S$ be the nondegenerate bilinear invariant form defined as $S(h,h)=2S(x,y)=2$, $S(h,x)=S(h,y)=0$.
Then $(\g,S)$ be the quadratic Lie algebra and an element $r=h\otimes x-x\otimes h$ is
a skew-symmetric solution of the classical Yang-Baxter equation in Lie algebra $sl(2,\mathbb C)$.
The corresponding map $B$ acts as
    \begin{equation}\label{RB-Scew}
    B(h)=2x,\ B(x)=0,\ B(y)=-h.
    \end{equation}
It is straightforward to check that the map $B$ satisfies  the condition \eqref{Reynolds-operator}, that is, a map $R=B$ is a skew-symmetric Reynolds operator. Obviously, $RB=BR$. Then $(\g,\g^*_r,R)$ is a Reynolds Lie bialgebra by Theorem \ref{thm:FL}.
\end{ex}

\begin{ex}
    Let $\g=sl_2(\mathbb C)$, $r=h\otimes x-x\otimes h$ be the skew-symmetric solution of the classical Yang-Baxter
    equation from the previous example and $(\g,\g^*_r)$ be
    the corresponding triangular Lie bialgebra. Note that the corresponding map $B$ defined by \eqref{RB-Scew}
     satisfies $B=-\ad_x$. Consider the Reynolds operator $R=(\Id-\ad_{h})^{-1}$ from Example \ref{dre-Rey-ex}
      (acting as $R(h)=h,\ R(x)=-x,\ R(y)=\frac{1}{3}y$). Then
    $$
    [B,d]=[\ad_{-x},\ad_{h}]=\ad_{[-x,h]}=2\ad_x=-2B.
    $$
    Thus, by Proposition \ref{inv-Der}, $R$ is a Reynolds operator on $(\g,B,S)$. Therefore, by Theorem \ref{thm:FL}, $(\g,\g_r,R)$ is a Reynolds Lie bialgebra.
\end{ex}

\section{The classical Yang-Baxter equation, relative Rota-Baxter operators and Reynolds pre-Lie algebras}\label{sec:YBE}
In this section, first we introduce the notion of the classical Yang-Baxter equation in a Reynolds Lie algebra, whose solution gives rise to a Reynolds Lie bialgebra. Then we introduce the notion of relative Rota-Baxter operators on a Reynolds Lie algebra with respect to a representation, which can give rise to solutions of the classical Yang-Baxter equation in the semidirect product Reynolds Lie algebra. Finally we introduce the notion of a Reynolds pre-Lie algebra, which can also produce  solutions of the classical Yang-Baxter equation in certain semidirect product Reynolds Lie algebra.

\subsection{The classical Yang-Baxter equation in a Reynolds Lie algebra}




Let $(\g,\Delta)$ be a Lie coalgebra. For any $x\in \g,$ we denote $\Delta(x)= x_{(1)}\otimes x_{(2)}.$
\begin{defi}
A Reynolds Lie coalgebra is a Lie coalgebra $(\g,\Delta)$ equipped with an operator $R$ such that
\begin{eqnarray}\label{Lie-coalgebra-Reynolds}
  (R\otimes R)\Delta=(R\otimes\Id+\Id\otimes R-R\otimes R)\Delta R.
\end{eqnarray}
\end{defi}
In Sweedler's notation, this means
\begin{eqnarray*}
  R(x_{(1)})\otimes R(x_{(2)})= R((Rx)_{(1)})\otimes (Rx)_{(2)}+ (Rx)_{(1)}\otimes R((Rx)_{(2)}) -R((Rx)_{(1)})\otimes R((Rx)_{(2)}).
\end{eqnarray*}
 We use the triple $(\g,\Delta,R)$ to denote a Reynolds Lie coalgebra. The following Proposition can be directly obtained.

\begin{pro}\label{prop52}
    A triple $(\g,\Delta,R)$ is  a Reynolds Lie coalgebra if and only if $(\g^*,[\cdot,\cdot]_{\g^*},R^*)$ is a Reynolds Lie algebra.
\end{pro}

\begin{defi}
  A Reynolds Lie bialgebra $(\g,\Delta,R)$ is called {\bf coboundary} if $\Delta$ is a $1$-coboundary of the Lie algebra $\g$ associated to the representation $\ad\otimes \Id+\Id\otimes\ad$ on the tensor space $\g\otimes \g$, that is, $\Delta$ is defined by \eqref{r-coboundary}, for some $r\in \g\otimes\g.$
\end{defi}

\begin{pro}\label{r-matrix-condition}
  Let $(\g,[\cdot,\cdot]_{\g},R)$ be a Reynolds Lie algebra and $r\in \g\otimes\g.$
  Suppose that the linear map $\Delta:\g\rightarrow\g\otimes\g$ defined by \eqref{r-coboundary}
  induces a Lie algebra structure on $\g^*$ such that $(\g,\g^*)$ is a Lie bialgebra.
   Then $(\g^*,[\cdot,\cdot]_{\g^*},{-R}^*)$ is a Reynolds Lie algebra if and only if the following equation holds:
  \begin{eqnarray} \label{Rey-r-matrix-1}
\nonumber &&\Big(\ad_{Rx}\otimes\Id+\Id\otimes \ad_{Rx}+R\circ\ad_{Rx}\otimes \Id+\Id\otimes R\circ \ad_{Rx}
 \\&&-R\circ \ad_x\otimes \Id-\Id\otimes R\circ\ad_x\Big)
 (R\otimes\Id+\Id\otimes R)(r)=0
  \end{eqnarray}
\end{pro}

\begin{proof} By Proposition \ref{prop52},
$(\g^*,[\cdot,\cdot]_{\g^*},{-R}^*)$ is a Reynolds Lie algebra if and only if $(\g,\Delta,-R)$ is a Reynolds Lie coalgebra.
By \eqref{Lie-coalgebra-Reynolds}, for all $x\in \g,$ we
 \begin{eqnarray*}
&&(R\otimes\Id)\Delta (Rx)+(\Id\otimes R)\Delta (Rx)+(R\otimes R)\Delta (Rx)-(R\otimes R)\Delta(x)\\
&=&(\uline{R\circ\ad_{Rx}\otimes\Id}+R\otimes\ad_{Rx}
+\ad_{Rx}\otimes R+\uuline{\Id\otimes R\circ\ad_{Rx}}\\
&&+R\circ\ad_{Rx}\otimes R+R\otimes R\circ\ad_{Rx}-R\circ\ad_x\otimes R-R\otimes R\circ\ad_x)(r)\\
&=&(\uline{\ad_{Rx}\circ R\otimes\Id-R\circ \ad_x\circ R\otimes \Id+R\circ\ad_{Rx}\circ R\otimes \Id}+R\otimes \ad_{Rx}\\
&&+\ad_{Rx}\otimes R+\uuline{\Id\otimes\ad_{Rx}\circ R-\Id\otimes R\circ \ad_{x}\circ R+\Id\otimes R\circ\ad_{Rx}\circ R}\\
&&+R\circ \ad_{Rx}\otimes R+R\otimes R\circ\ad_{Rx}-R\circ \ad_x\otimes R-R\otimes R\circ\ad_x)(r)\\
&=&(\ad_{Rx}\otimes \Id)(R\otimes \Id)(r)-(R\circ\ad_x\otimes \Id)(R\otimes \Id)(r)+(R\circ\ad_{Rx}\otimes\Id)(R\otimes\Id)(r)\\
&&+(\Id\otimes \ad_{Rx})(R\otimes \Id)(r)+(\ad_{Rx}\otimes \Id)(\Id\otimes R)(r)+(\Id\otimes \ad_{Rx})(\Id\otimes R)(r)\\
&&-(\Id\otimes R\circ \ad_x)(\Id\otimes R)(r)+(\Id\otimes R\circ\ad_{Rx})(\Id\otimes R)(r)+(R\circ \ad_{Rx}\otimes \Id)(\Id\otimes R)(r)\\
&&+(\Id\otimes R\circ\ad_{Rx})(R\otimes \Id)(r)-(R\circ\ad_x\otimes\Id)(\Id\otimes R)(r)-(\Id\otimes R\circ\ad_x)(R\otimes\Id)(r)\\
&=&(\ad_{Rx}\otimes\Id+\Id\otimes \ad_{Rx}+R\circ\ad_{Rx}\otimes \Id+\Id\otimes R\circ \ad_{Rx}\\
&&-R\circ \ad_x\otimes \Id-\Id\otimes R\circ\ad_x)(R\otimes\Id+\Id\otimes R)(r)\\
&=&0,
 \end{eqnarray*}
 which means $(\g^*,[\cdot,\cdot]_{\g^*},{-R}^*)$ is a Reynolds Lie algebra if and only if \eqref{Rey-r-matrix-1} holds.
\end{proof}
By Theorem \ref{r-bialgebra} and Proposition \ref{r-matrix-condition}, we have the following conclusion.
\begin{thm}\label{Rey-bialgebra-cybe}
Let $(\g,[\cdot,\cdot]_{\g},R)$ be a Reynolds Lie algebra and $r\in \g\otimes\g$.
Then the linear map $\Delta:\g\rightarrow\g\otimes\g$ defined by \eqref{r-coboundary} induces a Lie algebra structure on $\g^*$ such that $(\g,\Delta,R)$ is a Reynolds Lie bialgebra
if the following equations hold:
\begin{eqnarray}
(\ad_x\otimes \Id+\Id\otimes\ad_x\big)(r+\sigma(r))&=&0,\\
\label{LIE-Yang-Baxter}[[r,r]]&=&0,\\
\label{Rey-LIE-Yang-Baxter-2}(R\otimes\Id+\Id\otimes R)(r)&=&0.
\end{eqnarray}
where $[[r,r]]$ is defined by \eqref{r-matrix}.
\end{thm}

\begin{defi}
The equation \eqref{LIE-Yang-Baxter} together with conditions \eqref{Rey-LIE-Yang-Baxter-2} is called the {\bf classical Yang-Baxter equation} in the Reynolds Lie algebra $(\g,[\cdot,\cdot]_{\g},R)$.
\end{defi}

\begin{rmk}
    Note that if $(\g,S)$ is a quadratic Lie algebra and $r\in \g\otimes \g$, then the condition \eqref{Rey-LIE-Yang-Baxter-2} is equivalent to the condition $R\circ B =-B\circ R^*$, where $B$ is the corresponding to $r$ map defined by \eqref{sharp-condition}.
\end{rmk}

\subsection{Relative Rota-Baxter operators on Reynolds Lie algebras}

We introduce the notion of a relative Rota-Baxter operator on a Reynolds Lie algebra.

\begin{defi}
Let $(W;T,\rho)$ be a representation of a Reynolds Lie algebra $(\g,[\cdot,\cdot]_{\g},R).$ A linear map $K:W\rightarrow\g$ is called a {\bf relative Rota-Baxter operator} on $(\g,[\cdot,\cdot]_{\g},R)$ with
respect to the representation $(W;T,\rho)$ if the following equation holds:
\begin{eqnarray}
  \label{relative-RBO}[Ku,Kv]_{\g}&=&K(\rho(Ku)v-\rho(Kv)u),\quad \forall u,v\in W,\\
  \label{compatible-condition}R\circ K&=&K\circ T.
\end{eqnarray}
Moreover, a relative Rota-Baxter operator on a Reynolds Lie algebra with
respect to the adjoint representation $(\g;R,\ad)$ is called a {\bf Rota-Baxter operator of weight $0$.}
\end{defi}
\begin{rmk}
  Note that a linear map $K:W\rightarrow\g$ satisfying \eqref{relative-RBO} is called a relative
   Rota-Baxter operator on $(\g,[\cdot,\cdot]_{\g})$ with
respect to the representation $(W;\rho)$ \rm{(\cite{Ku})}.
\end{rmk}

\begin{pro}\label{descendent-AveLie}
  Let $K:W\rightarrow\g$ be a relative Rota-Baxter operator on a Reynolds Lie algebra $(\g,[\cdot,\cdot]_{\g},R)$  with
respect to the representation $(W;T,\rho).$ Define
\begin{eqnarray}\label{relative-Rotabaxter-braket}
 [u,v]_{K}=\rho(Ku)v-\rho(Kv)u,\quad \forall u,v\in W.
\end{eqnarray}
 Then $(W,[\cdot,\cdot]_{K},T)$ is a Reynolds Lie algebra, called the {\bf descendent Reynolds Lie algebra.}
\end{pro}
\begin{proof}
  It was proved in \rm{(\cite{Ku})} that if $K:W\rightarrow\g$ is a relative Rota-Baxter operator on a Lie algebra
$\g$ with respect to  a representation $(W;\rho),$ then $(W,[\cdot,\cdot]_{K})$ is a Lie algebra. By \eqref{representation-Reynolds-Liealg} and \eqref{compatible-condition}, for all $u,v\in W,$ we have
\begin{eqnarray*}
[Tu,Tv]_{K}&=&\rho(K(Tu))(Tv)-\rho(K(Tv))(Tu)\\
&=&\rho(R(Ku))(Tv)-\rho(R(Kv))(Tu)\\
&=&T(\rho(Ku)(Tv))+T(\rho(R(Ku)v))-T(\rho(R(Ku))(Tv))\\
&&-T(\rho(Kv)(Tu))-T(\rho(R(Kv)u))+T(\rho(R(Kv))(Tu))\\
&=&T([Tu,v]_{K}+[u,Tv]_{K}-[Tu,Tv]_{K}),
\end{eqnarray*}
which implies that $T:W\rightarrow W$ is a Reynolds operator on the Lie algebra $(W,[\cdot,\cdot]_{K}).$
\end{proof}

\begin{cor}
 Let $K:W\rightarrow\g$ be a relative Rota-Baxter operator on a Reynolds Lie algebra $(\g,[\cdot,\cdot]_{\g},R)$ with
respect to the representation $(W;T,\rho).$ Then $K$ is a homomorphism from the descendent Reynolds Lie algebra $(W,[\cdot,\cdot]_{K},T)$ to the Reynolds Lie algebra $(\g,[\cdot,\cdot]_{\g},R).$
\end{cor}

\begin{pro}\label{construction-mathched-pair}
Let $K:W\rightarrow\g$ be a relative Rota-Baxter operator on a Reynolds Lie algebra $(\g,[\cdot,\cdot]_{\g},R)$ with
respect to the representation $(W;T,\rho).$ Then $((\g,R),(W,T);\rho,\mu)$
is a matched pair of Reynolds Lie algebras, where $(W,T)$ is the descendent Reynolds Lie algebra, the Lie bracket $[\cdot,\cdot]_{K}$ is defined by \eqref{relative-Rotabaxter-braket} and $\mu:W\rightarrow\gl(\g)$ is given by
\begin{eqnarray*}
\mu(u)x=K(\rho(x)u)-[x,Ku]_{\g},\quad \forall x\in \g,u\in W.
\end{eqnarray*}
\end{pro}
\begin{proof}
By \eqref{relative-RBO} and \eqref{compatible-condition}, for all $x\in \g, u,v\in W,$ we have $[\mu(u),\mu(v)]x=\mu([u,v]_{K})x,$ and
\begin{eqnarray*}
\mu(Tu)(Rx)&=&K(\rho(Rx)(Tu))-[Rx,K(Tu)]_{\g}\\
&=&K(T(\rho(x))(Tu))+K(T(\rho(Rx)u))-K(T(\rho(Rx)(Tu)))\\
&&-R[Rx,Ku]_{\g}-R[x,R(Ku)]_{\g}+R[Rx,R(Ku)]_{\g}\\
&=&R(\mu(u)(Rx)+\mu(Tu)(x)-\mu(Tu)(Rx))
\end{eqnarray*}
which implies that $(\g;R,\mu)$  is a representation of the descendent Reynolds Lie algebra $(W,[\cdot,\cdot]_{K},T)$.  It is also known that $((\g,[\cdot,\cdot]_{\g}),(W,[\cdot,\cdot]_{K}),\rho,\mu)$ is a matched pair of Lie algebras.
Therefore, $((\g,R),(W,T);\rho,\mu)$
is a matched pair of Reynolds Lie algebras.
\end{proof}

We have the following relationship between relative Rota-Baxter operators on Reynolds Lie algebras and solutions of the classical Yang-Baxter equation in Reynolds Lie algebras.
\begin{pro}\label{pro:rrr}
Let $(\g,[\cdot,\cdot]_{\g},R)$ be a Reynolds Lie algebra. Suppose that $r\in \g\otimes\g$ is skew-symmetric. Then $r_+:\g^*\rightarrow\g$ is a relative Rota-Baxter operator on $(\g,[\cdot,\cdot]_{\g},R)$
with respect to the coadjoint representation $(\g^*;-R^*,\ad^*)$ if and only if $r$ is a solution of the
classical Yang-Baxter equation in the Reynolds Lie algebra $(\g,[\cdot,\cdot]_{\g},R).$
\end{pro}
\begin{proof}
 It is well known that  $r_+:\g^*\rightarrow\g$ is a relative Rota-Baxter operator on a Lie algebra
$\g$ with respect to the representation $(\g^*;\ad^*)$, i.e.
  \begin{equation*}
    [r_{+}(\xi),r_+(\eta)]_\g=r_+({\ad}^*_{r_+(\xi)}\eta-\ad^*_{r_+(\eta)}\xi), \quad \forall \xi, \eta\in \g^*,
\end{equation*}
if and only if $r$ is a solution of the classical Yang-Baxter equation in the Lie algebra $\g.$
By  $R\circ r_+=r_+\circ (-R^*),$  we have
 \begin{eqnarray}
   (R\otimes\Id)r=-(\Id\otimes R)r,
 \end{eqnarray}
which implies \eqref{Rey-LIE-Yang-Baxter-2} holds. Then $r_+:\g^*\rightarrow\g$ is a relative Rota-Baxter operator on $(\g,[\cdot,\cdot]_{\g},R)$ with respect to the coadjoint representation $(\g^*;-R^*,\ad^*)$ if and only if $r$ is a solution of the
classical Yang-Baxter equation in the Reynolds Lie algebra $(\g,[\cdot,\cdot]_{\g},R).$
This completes the proof.
\end{proof}

In a quadratic Reynolds Lie algebra, relative Rota-Baxter operators with respect to the coadjoint representation are equivalent to Rota-Baxter operators of weight 0.

\begin{pro}
Let $(\g,R,S)$ be a quadratic Reynolds Lie algebra and $K:\g^*\rightarrow\g$ a linear map. Then $K$ is a relative Rota-Baxter operator on $(\g,[\cdot,\cdot]_{\g},R)$ with respect to the coadjoint representation
$(\g^*;-R^*,\ad^*)$ if and only if $K\circ S^{\sharp}$ is a Rota-Baxter operator of weight $0$ on the Reynolds Lie algebra $(\g,[\cdot,\cdot]_{\g},R).$
\end{pro}
\begin{proof}
  For all $x,y \in \g,$ by Theorem \ref{quadratic-Rep}, we have
\begin{eqnarray*}
(K\circ S^{\sharp})([K\circ S^{\sharp}(x),y]_{\g}+[x,K\circ S^{\sharp}(y)]_{\g})
&=&(K\circ S^{\sharp})(\ad_{K\circ  S^{\sharp}(x)}y-\ad_{K\circ  S^{\sharp}(y)}x)\\
&=&K\circ\ad^*_{K\circ  S^{\sharp}(x)}S^{\sharp}(y)-K\circ\ad^*_{K\circ  S^{\sharp}(y)}S^{\sharp}(x),
\end{eqnarray*}
which implies that $K$ is a relative Rota-Baxter operator if and only if $K\circ S^\sharp$ is a Rota-Baxter operator of weight $0$ on the Lie algebra $\g$.

On the other hand, by \eqref{eq:SP-1},
 $R\circ(K\circ S^{\sharp})=(K\circ S^{\sharp}) \circ R$ if and only if $(R\circ K)\circ S^{\sharp}=-(K\circ R^*)\circ S^{\sharp},$ which is equivalent to $R\circ K=K\circ (-R^*)$ due to the fact that the bilinear form $S$ is nondegenerate.
 Thus, $K$ is a relative Rota-Baxter operator if and only if $K\circ S^{\sharp}$ is a Rota-Baxter operator of weight $0$ on the Reynolds Lie algebra $(\g,[\cdot,\cdot]_{\g},R).$
\end{proof}

\begin{cor}
Let $(\g,R,S)$ be a quadratic Reynolds Lie algebra.  Suppose that $r\in \g\otimes\g$ is skew-symmetric. Then $r_+\circ S^{\sharp}$ is a Rota-Baxter operator of weight $0$
on the Reynolds Lie algebra $(\g,[\cdot,\cdot]_{\g},R)$ if and only if $r$ is a solution of the
classical Yang-Baxter equation in the Reynolds Lie algebra $(\g,[\cdot,\cdot]_{\g},R).$
\end{cor}

Recall that a relative Rota-Baxter operator on Lie algebras can give rise to a solution of the classical Yang-Baxter equation in the semi-direct product Lie algebra as follows.

\begin{lem}\label{lem:rt}{\rm (\cite{Bai})}
Let $(\g,[\cdot,\cdot]_{\g})$ be a Lie algebra and $(W;\rho)$ be a representation. Let $K: W\rightarrow\g$ be a linear map which is identified as an element in
$(\g\ltimes_{\rho^*} W^*)\otimes (\g\ltimes_{\rho^*} W^*)$
 $($through ${\rm Hom}(W,\g)\cong W^*\otimes \g\subseteq (\g\ltimes_{\rho^*} W^*)\otimes (\g\ltimes_{\rho^*} W^*)$$)$.
Then
\begin{equation}
r_K:=K-\sigma(K)
\end{equation}
is a skew-symmetric  solution of the
classical Yang-Baxter equation
in the Lie algebra $\g\ltimes_{\rho^*} W^*$ if and only if $K$
is a relative Rota-Baxter operator on $\g$ with respect to $(W;\rho)$.
\end{lem}

Let $(\g,[\cdot,\cdot]_{\g},R)$ be a Reynolds Lie algebra and  $(W;T,\rho)$ be a representation. Then there is a semi-direct product Reynolds Lie algebra $(\g \ltimes_{\rho^*}W^*,R-T^*)$.
Any linear map $K:W\rightarrow\g$ can be viewed as an element $\bar{K}\in \otimes^2(\g\oplus W^*)$ via
\begin{eqnarray}
  \bar{K}(\xi+u,\eta+v)=\langle Ku,\eta\rangle,\quad\forall \xi+u,\eta+v\in \g^*\oplus W.
\end{eqnarray}

\begin{thm}\label{skew-solution-Reynolds}
  Let $K:W\rightarrow\g$ be a relative Rota-Baxter operator on a Reynolds Lie algebra $(\g,[\cdot,\cdot]_{\g},R)$ with respect to  a representation $(W;T,\rho)$. Then
$r_K=\bar{K}-\sigma(\bar{K})$ is a skew-symmetric solution of the classical Yang-Baxter equation in the semi-direct product Reynolds Lie algebra $(\g\ltimes_{\rho^*}W^*,R-T^*)$.
\end{thm}
\begin{proof}
By Lemma \ref{lem:rt},
  $r_K$ satisfies the classical Yang-Baxter equation in the semi-direct product Lie algebra $\g\ltimes_{\rho^*}W^*$ if and only if $K:W\to\g$ is a relative Rota-Baxter operator on the Lie algebra $\g$ with respect to the representation $(W;\rho)$. Thus, we only need to show that $r_K$ satisfies \eqref{Rey-LIE-Yang-Baxter-2} for the Reynolds operator $R-T^*.$ By \eqref{compatible-condition}, we have
  \begin{eqnarray*}
   && \Big((R-T^*)\otimes \Id+\Id\otimes (R-T^*)\Big) \bar{K} (\xi+u,\eta+v)\\
    &=& \bar{K} ((R^*-T)(\xi+u),\eta+v)+ \bar{K} (\xi+u,(R^*-T)(\eta+v))\\
    &=& \bar{K} (R^*\xi-T u,\eta+v)+ \bar{K} (\xi+u,R^*\eta-T v)\\
    &=&\langle -KTu,\eta\rangle+\langle Ku,R^*\eta\rangle\\
    &=&\langle (-K\circ T+R\circ K)u,\eta\rangle\\
    &=&0.
  \end{eqnarray*}
Similarly, we have
  \begin{eqnarray*}
   && \Big((R-T^*)\otimes \Id+\Id\otimes (R-T^*)\Big) \sigma(\bar{K}) (\xi+u,\eta+v)\\
   &=&-\sigma\Big(\big((R-T^*)\otimes \Id+\Id\otimes (R-T^*)\big) \bar{K} \Big) (\xi+u,\eta+v)\\
   &=&- \Big(\big((R-T^*)\otimes \Id+\Id\otimes (R-T^*)\big) \bar{K} \Big) (\eta+v,\xi+u)\\
       &=&0,
  \end{eqnarray*}
  which implies that
  $$
   \Big((R-T^*)\otimes \Id+\Id\otimes (R-T^*)\Big)r_{K}=0.
  $$
Hence \eqref{Rey-LIE-Yang-Baxter-2} holds. Therefore, $r_K$ is a skew-symmetric solution of the classical Yang-Baxter equation in the Reynolds Lie algebra $(\g\ltimes_{\rho^*}W^*,R-T^*)$.
\end{proof}

\subsection{Reynolds pre-Lie algebras}

A {\bf pre-Lie algebra} $(\g,\{\cdot,\cdot\})$ is a vector space $\g$ equipped with a bilinear bracket $\{\cdot,\cdot\}:\g\otimes\g\rightarrow\g$
such that for any $x,y,z\in \g,$ the associator $(x,y,z)=\{\{x,y\},z\}-\{x,\{y,z\}\}$ is symmetric in $x,y,$ i.e. $(x,y,z)=(y,x,z),$ or equivalently
\begin{eqnarray}\label{pre-Lie-defi}
\{\{x,y\},z\}-\{x,\{y,z\}\}=\{\{y,x\},z\}-\{y,\{x,z\}\}.
\end{eqnarray}

\begin{defi}
Let $(\g,\{\cdot,\cdot\})$ be a pre-Lie algebra. A linear operator $R:\g\rightarrow \g$ is called
{\bf a Reynolds operator} if $R$ satisfies
   \begin{eqnarray}\label{pre-Reynolds-operator}
 \{Rx,Ry\}=R\Big(\{Rx,y\}+\{x,Ry\}-\{Rx,Ry\}\Big), \quad \forall x, y \in \g.
\end{eqnarray}
Moreover, a pre-Lie algebra $(\g,\{\cdot,\cdot\})$ equipped with a Reynolds operator $R$ is called a {\bf Reynolds pre-Lie algebra},
which is denoted by $(\g,\{\cdot,\cdot\},R).$
\end{defi}

Then following conclusion is obvious.
\begin{pro}\label{Reynolds-Pre-conclusion3}
Let $(\g,\{\cdot,\cdot\},R)$ be a Reynolds pre-Lie algebra. Then
\begin{itemize}
  \item[{\rm (i)}]  $(\g, \Courant{\cdot,\cdot}_{\g},R)$
is a Reynolds Lie algebra, where the bracket $\Courant{\cdot,\cdot}_{\g}$ is given by
\begin{eqnarray*}
\Courant{x,y}_{\g}=\{x,y\}-\{y,x\},\quad \forall x,y\in \g.
\end{eqnarray*}
We call  $(\g,\Courant{\cdot,\cdot}_{\g},R)$ is the {\bf sub-adjacent Reynolds Lie algebra} of  $(\g,\{\cdot,\cdot\},R)$, and $(\g,\{\cdot,\cdot\},R)$ is called a {\bf compatible Reynolds pre-Lie algebra} of $(\g, \Courant{\cdot,\cdot}_{\g},R).$
 \item[{\rm (ii)}] Define a linear map $L:\g\rightarrow \gl(\g)$ by
$ L(x)y=\{x,y\},$ for all $x,y \in \g.$ Then $(\g;R,L)$ is a representation of the  Reynolds Lie algebra $(\g, \Courant{\cdot,\cdot}_{\g},R)$.
\item[{\rm (iii)}] The identity map $\Id:\g\rightarrow\g$ is a relative Rota-Baxter operator on the Reynolds Lie algebra $(\g,\Courant{\cdot,\cdot}_{\g},R)$ with respect to the representation $(\g;R,L).$
\end{itemize}
\end{pro}
\begin{proof}{\rm (i)}It is known that $(\g,\Courant{\cdot,\cdot}_{\g})$ is a Lie algebra. For all $x,y\in \g,$ we have
 \begin{eqnarray*}
 \Courant{Rx,Ry}_{\g}&=&\{Rx,Ry\}-\{Ry,Rx\}\\
 &=&R(\{Rx,y\}+\{x,Ry\}-\{Rx,Ry\})-R(\{Ry,x\}+\{y,Rx\}-\{Ry,Rx\})\\
 &=&R(\Courant{Rx,y}_{\g}+\Courant{x,Ry}_{\g}-\Courant{Rx,Ry}_{\g}).
 \end{eqnarray*}
 Therefore, $(\g, \Courant{\cdot,\cdot}_{\g},R)$ is a Reynolds Lie algebra.
 \item[{\rm (ii)}] For all $x,y,z\in\g,$ by \eqref{pre-Lie-defi},we have $L(\Courant{x,y}_{\g})(z)=[L(x),L(y)](z)$ and
  \begin{eqnarray*}
 L(Rx)(Ry)=\{Rx,Ry\}&=&R(\{Rx,y\}+\{x,Ry\}-\{Rx,Ry\})\\
 &=&R(L(x)(Ry)+L(Rx)(y)-L(Rx)(Ry)),
  \end{eqnarray*}
  which implies that $(\g;R,L)$ is a representation of the  Reynolds Lie algebra $(\g, \Courant{\cdot,\cdot}_{\g},R)$.
\item[{\rm (iii)}]It follows from the definition of the sub-adjacent Reynolds Lie algebra directly.
\end{proof}

\begin{pro}\label{RRB-Reynolds-Pre-Lie}
Let $K:W\rightarrow\g$ be a relative Rota-Baxter operator on a Reynolds Lie algebra $(\g,[\cdot,\cdot]_{\g},R)$   with respect to the representation $(W;T,\rho).$
Then $(W,\{\cdot,\cdot\}_{K},T)$ is  a Reynolds pre-Lie algebra, where $\{\cdot,\cdot\}_{K}$ is given by
\begin{eqnarray}\label{Rey-pre}
\{u,v\}_{K}=\rho(Ku)v,\quad \forall u,v\in W.
\end{eqnarray}
\end{pro}
\begin{proof}
It is well known that if $K$ is a relative Rota-Baxter operator on a Lie algebra $(\g,[\cdot,\cdot]_{\g})$ with respect to a representation $(W;\rho),$ then $(W,\{\cdot,\cdot\}_{K})$ is a pre-Lie algebra.

By $R\circ K=K\circ T$ and the fact that $(W;T,\rho)$ is a representation, for all $u,v\in W,$ we have
\begin{eqnarray*}
\{Tu,Tv\}_{K}=\rho(K(Tu))(Tv)&=&\rho(R(Ku))(Tv)\\
&=&T(\rho(Ku)(Tv))+T(\rho(R(Ku))v)-T(\rho(R(Ku)(Tv)))\\
&=&T(\{u,Tv\}_{K}+\{Tu,v\}_{K}-\{Tu,Tv\}_{K}).
\end{eqnarray*}
Hence $T$ is a Reynolds operator on the pre-Lie algebra $(W,\{\cdot,\cdot\}_{K}).$
\end{proof}

\begin{pro}
  Let $(\g,[\cdot,\cdot]_{\g},R)$ be a Reynolds Lie algebra. Then there exists a compatible
  Reynolds pre-Lie algebra $(\g,\{\cdot,\cdot\},R)$
  if and only if there exists a representation $(W;T,\rho)$ of the Reynolds
  Lie algebra $(\g,[\cdot,\cdot]_{\g},R)$ and an invertible relative
  Rota-Baxter operator $K:W\rightarrow\g$ on $(\g,[\cdot,\cdot]_{\g},R)$
  with respect to the representation $(W;T,\rho)$. Furthermore, the compatible Reynolds pre-Lie algebra structure $(\g,\{\cdot,\cdot\}_{K},R)$ on $\g$ is given by
\begin{eqnarray*}
  \{x,y\}_{K}:=K(\rho(x)K^{-1}y),\quad \forall x,y\in \g.
\end{eqnarray*}
\end{pro}
\begin{proof}
Let $K:W\rightarrow\g$ be an invertible relative Rota-Baxter operator on
$(\g,[\cdot,\cdot]_{\g},R)$ with respect to a representation $(W;T,\rho).$ By Proposition \ref{RRB-Reynolds-Pre-Lie},
there exist a Reynolds pre-Lie algebra $(W,\{\cdot,\cdot\}_{K},T)$ on $W,$ where $\{\cdot,\cdot\}_{K}$ is given by \eqref{Rey-pre}. Since $K$ is an invertible relative Rota-Baxter operator, for all $x,y\in \g,$
we obtain that
\begin{eqnarray*}
\{x,y\}_{K}:=K\{K^{-1}x,K^{-1}y\}_{K}=K(\rho(x)K^{-1}y)
\end{eqnarray*}
is a pre-Lie algebra on $\g.$ Moreover, we have
\begin{eqnarray*}
\{Rx,Ry\}_{K}&=&K(\rho(Rx)(K^{-1}(Ry)))\\
&=&K(\rho(Rx)(T(K^{-1}y)))\\
&=&K\Big(T(\rho(x)(T(K^{-1}y)))+T(\rho(Rx)(K^{-1}y))-T(\rho(Rx)(T(K^{-1}y)))\Big)\\
&=&R(\{Rx,y\}_{K}+\{x,Ry\}_{K}-\{Rx,Ry\}_{K}).
\end{eqnarray*}
Thus, $(\g,\{\cdot,\cdot\}_{K},R)$ is a Reynolds pre-Lie algebra, whose subadjacent Reynolds Lie algebra is exactly $(\g,[\cdot,\cdot]_{\g},R)$.

On the other hand, let $(\g,\{\cdot,\cdot\},R)$ be a compatible Reynolds pre-Lie algebra of the Reynolds Lie algebra $(\g,[\cdot,\cdot]_{\g},R).$ By Proposition \ref{Reynolds-Pre-conclusion3}, $(\g;R,L)$ is a representation of the Reynolds Lie algebra $(\g,[\cdot,\cdot]_{\g},R).$ Moreover, the identity map $\Id:\g\rightarrow\g$ is a relative Rota-Baxter operator on  $(\g,[\cdot,\cdot]_{\g},R)$ with respect to the representation $(\g;R,L).$ The proof is finished.
\end{proof}

\begin{thm}
Let $(\g,\{\cdot,\cdot\},R)$ be a Reynolds pre-Lie algebra. Then
\begin{eqnarray*}
  r:=\sum^{n}_{i=1}(e_i\otimes e_i^*-e_i^*\otimes e_i)
\end{eqnarray*}
is a skew-symmetric solution of the classical Yang-Baxter equation in the semi-direct product Reynolds Lie algebra $(\g\ltimes_{L^*}\g^*,R-R^*),$ where $\{e_1,\cdots,e_n\}$ is a basis of $\g$ and $\{e^*_1,\cdots,e^*_n\}$ is its dual basis.
\end{thm}
\begin{proof}
  Since $(\g,\{\cdot,\cdot\},R)$ \mg{is} a Reynolds pre-Lie algebra, then $\Id:\g\rightarrow\g$ is a relative Rota-Baxter operator on the sub-adjacent Reynolds Lie algebra $(\g,\Courant{\cdot,\cdot}_{\g},R)$ with respect to the representation $(\g;R,L).$ By Theorem \ref{skew-solution-Reynolds}, $r=\sum^{n}_{i=1}(e_i\otimes e_i^*-e_i^*\otimes e_i)$ is a skew-symmetric solution of the classical Yang-Baxter equation in the semi-direct product Reynolds Lie algebra $(\g\ltimes_{L^*}\g^*,R-R^*).$
\end{proof}

\begin{ex}
Consider a $2$-dimensional pre-Lie algebra  $(\g,\{\cdot,\cdot\})$ from \rm{(\cite{Bai-Meng})}. It has a basis $\{e_1,e_2\}$ with the following table of multiplication:
  \begin{eqnarray*}
   \{e_1,e_2\} =e_1,\quad \{e_2,e_2\} =e_2,\quad \{e_1,e_1\}=\{e_2,e_1\}=0.
  \end{eqnarray*}
It is straightforward to check that for any $\alpha,\beta\in \mathbb F$, a map $R:\g\rightarrow \g$ defined as $R(e_1)=\alpha e_1$, $R(e_2)=\beta e_1+e_2$ is a Reynolds operator on the pre-Lie algebra $(\g,\{\cdot,\cdot\})$.
By Proposition \ref{Reynolds-Pre-conclusion3}, we have the sub-adjacent Reynolds  Lie algebra  $(\g,\Courant{\cdot,\cdot}_{\g},R)$, where non-zero products are given by
$
   \Courant{e_1,e_2}_{\g}=e_1.
$
The identity map $\Id:\g\rightarrow\g$ is a relative Rota-Baxter operator on $(\g,\Courant{\cdot,\cdot}_{\g},R)$  with respect to the representation $(\g;R,L),$
where the  map $L:\g\rightarrow\gl(\g)$ is defined by
$$
L_{e_1}(e_1)=0,\ L_{e_1}(e_2)=e_1,\ L_{e_2}(e_1)=0,\ L_{e_2}(e_2)=e_2.
$$
Let $\{e_1^*,e_2^*\}$ be the dual basis of the dual vector space $\g^*.$
Then the dual representation $L^*:\g\rightarrow\gl(\g^*)$  with respect to the basis $\{e_1^*, e_2^*\}$ is given by
$$
L^*_{e_1}(e_1^*)=-e_2^*,\ L^*_{e_1}(e_2^*)=0,\ L^*_{e_2}(e_1^*)=0,\ L^*_{e_2}(e_2^*)=-e_2^*.
$$
Consider the dual representation $(\g^*;R^*,L^*)$ of the Reynolds Lie algebra $(\g,\Courant{\cdot,\cdot}_{\g},R),$ where $R^*$ is given by $R^*(e_1^*)=\alpha e_1^*+\beta e_2^*, R^*(e_2^*)=e_2^*.$
By Proposition \ref{semi-direct-ReyLiealge}, $(\g\ltimes_{L^*}\g^*,\Courant{\cdot,\cdot}_{\ltimes},R-R^*)$ is a $4$-dimensional Reynolds Lie algebra, where  the table of multiplication  with respect to the basis $\{e_1,e_2,e_1^*,e_2^*\}$ is
$\left(\begin{array}{cccc}
0&e_1&-e_2^*&0\\
-e_1&0&0&-e_2^*\\
e_2^*& 0&0&0\\
0&e_2^*&0&0
\end{array}\right)$.
Then $r=\sum^{2}_{i=1}(e_i\otimes e_i^*-e_i^*\otimes e_i)$  is a skew-symmetric solution of the classical Yang-Baxter equation in the above Reynolds Lie algebra.
By Theorem \ref{Rey-bialgebra-cybe}, there is a Reynolds Lie bialgebra $(\g\ltimes_{L^*}\g^*,\Delta,R-R^*),$ where
\begin{equation*}
\begin{matrix}
\begin{aligned}
\Delta(e_1)&=0,\\
\Delta(e_1^*)&=e_2^*\otimes e_1^*+e_1^*\otimes e_2^*,\\
(R-R^*)(e_1)&=a_{11}e_1,\\
(R-R^*)(e_1^*)&=-a_{11}e_1^*-a_{12}e_2^*,
\end{aligned}
&
\begin{aligned}
\quad \Delta(e_2)&=e_2^*\otimes e_2+e_1^*\otimes e_1-e_2\otimes e_2^*-e_1\otimes e_1^*,\\
\quad \Delta(e_2^*)&=2(e_2^*\otimes e_2^*),\\
\quad (R-R^*)(e_2)&=a_{12}e_1+e_2,\\
\quad (R-R^*)(e_2^*)&=-e_2^*.
\end{aligned}
\end{matrix}
\end{equation*}
\end{ex}

{\bf Acknowledgements.} The first author was supported by NSFC (12301034, 12461004, W2412041). The second author was supported by the Russian Science Foundation (project 25-41-
00005).  We give warmest thanks to Professor Yunhe Sheng for valuable suggestions and  helpful comments.

\end{document}